\documentclass[preprint]{imsart}
\usepackage[english]{babel}
\usepackage[utf8]{inputenc}
\usepackage{lmodern}
\usepackage[T1]{fontenc}
\usepackage{calc}

\usepackage{amsmath}
\usepackage{amsfonts}
\usepackage{amssymb}
\usepackage{amsthm}
\usepackage{euscript}
\usepackage[numbers]{natbib}
\usepackage{xspace}
\usepackage{bm}
\usepackage{enumerate}
\usepackage{hyperref}

\startlocaldefs

\newtheorem{definition}{Definition}[section]

\newtheorem{lemma}[definition]{Lemma}
\newtheorem{theorem}{Theorem}[section]
\newtheorem{proposition}[definition]{Proposition}

\theoremstyle{definition}
\newtheorem{example}{Example}[section]
\newtheorem{remark}[definition]{Remark}

\numberwithin{equation}{section}

\newcommand{\Ai}{\ensuremath{\mathcal A}\xspace}
\newcommand{\Bi}{\ensuremath{\mathcal B}\xspace}
\newcommand{\Ci}{\ensuremath{\mathcal C}\xspace}

\newcommand{\Ei}{\ensuremath{\mathcal E}\xspace}
\newcommand{\Fi}{\ensuremath{\mathcal F}\xspace}
\newcommand{\Gi}{\ensuremath{\mathcal G}\xspace}
\newcommand{\Hi}{\ensuremath{\mathcal H}\xspace}

\newcommand{\Pii}{\ensuremath{\mathcal P}\xspace}
\newcommand{\Qi}{\ensuremath{\mathcal Q}\xspace}

\newcommand{\Ti}{\ensuremath{\mathcal T}\xspace}

\newcommand{\SaS}{\ensuremath{\mathcal{S}\alpha\mathcal{S}}\xspace}

\newcommand{\Aiu}{\ensuremath{\mathcal A(u)}\xspace}

\newcommand{\Cio}{\ensuremath{\mathcal C_0}\xspace}

\newcommand{\zset}{\ensuremath{{\emptyset'}}\xspace}
\newcommand{\vset}{\ensuremath{{\emptyset}}\xspace}

\newcommand{\Cp}[1]{\ensuremath{{ \bm{\mathcal{A}_{#1}}} }\xspace}
\newcommand{\Gc}[1]{\ensuremath{{ \mathcal{G}_{\hspace{-0.4pt}#1}^{\hspace{0.5pt}*} }}\xspace}
\newcommand{\GcA}[1]{\ensuremath{{ \widetilde{\mathcal{G}}_{\hspace{-0.4pt}#1}^{\hspace{0.5pt}*} }}\xspace}

\newcommand{\pc}[1]{\ensuremath{{ p_{\hspace{-0.2pt}#1} }}\xspace}
\newcommand{\Pc}[1]{\ensuremath{{ P_{\hspace{-0.2pt}#1} }}\xspace}
\newcommand{\Pce}[2]{\ensuremath{{ P^{\hspace{0.2pt}#2}_{\hspace{-0.2pt}#1} }}\xspace}

\newcommand{\Cpi}[2]{\ensuremath{{ U_{\hspace{-1pt}#1}^{\hspace{0.2pt}#2}} }\xspace}
\newcommand{\eCpi}[1]{ \varepsilon_{#1} }

\newcommand{\vXCp}[1]{\ensuremath{ \bm{\mathbf{X}_{#1}} }\xspace}
\newcommand{\vxCp}[1]{\ensuremath{ \bm{\mathbf{x}_{#1}} }\xspace}

\newcommand{\pth}[1]{(#1)}
\newcommand{\pthb}[1]{\bigl(#1\bigr)}
\newcommand{\pthB}[1]{\Bigl(#1\Bigr)}
\newcommand{\pthbb}[1]{\biggl(#1\biggr)}

\newcommand{\bkt}[1]{[#1]}
\newcommand{\bktb}[1]{\bigl[#1\bigr]}
\newcommand{\bktB}[1]{\Bigl[#1\Bigr]}
\newcommand{\bktbb}[1]{\biggl[#1\biggr]}
\newcommand{\bktBB}[1]{\Biggl[#1\Biggr]}

\newcommand{\brc}[1]{\{#1\}}
\newcommand{\brcb}[1]{\bigl\{#1\bigr\}}

\newcommand{\brcbb}[1]{\biggl\{#1\biggr\}}

\newcommand{\dt}{\ensuremath{\mathrm d}\xspace} 
\newcommand{\disy}{\Delta} 
\newcommand{\eqdef}{:=}
\newcommand{\eqd}{\overset{\mathrm{d}}{=}}

\newcommand{\scpr}[2]{\left\langle #1,#2 \right\rangle}

\newcommand{\ivoo}[1]{\ensuremath{\left(#1\right)}}
\newcommand{\ivof}[1]{\ensuremath{\left(#1\right]}}
\newcommand{\ivfo}[1]{\ensuremath{\left[#1\right)}}
\newcommand{\ivff}[1]{\ensuremath{\left[#1\right]}}

\newcommand{\abs}[1]{\lvert#1\rvert}
\newcommand{\absb}[1]{\bigl\lvert#1\bigr\rvert}

\newcommand{\absbb}[1]{\biggl\lvert#1\biggr\rvert}

\newcommand{\norm}[1]{\lVert#1\rVert}
\newcommand{\normb}[1]{\bigl\lVert#1\bigr\rVert}

\renewcommand{\Pr}{\ensuremath{\mathbb P}\xspace}

\newcommand{\pr}[2][]{\mathbb{P}#1\pth{#2}}
\newcommand{\prb}[2][]{\mathbb{P}#1\pthb{\hspace{1pt}#2\hspace{1pt}}}

\newcommand{\prc}[3][]{\mathbb{P}#1\pth{#2\hspace{1.5pt}|\hspace{1.5pt}#3}}
\newcommand{\prcb}[3][]{\mathbb{P}#1\pthb{\hspace{1pt}#2\bigm|#3\hspace{1pt}}}

\newcommand{\esp}[2][]{\mathbb{E}#1\bkt{#2}}
\newcommand{\espb}[2][]{\mathbb{E}#1\bktb{\hspace{1pt}#2\hspace{1pt}}}
\newcommand{\espB}[2][]{\mathbb{E}#1\bktB{#2}}
\newcommand{\espbb}[2][]{\mathbb{E}#1\bktbb{#2}}
\newcommand{\espBB}[2][]{\mathbb{E}#1\bktBB{#2}}

\newcommand{\espc}[3][]{\mathbb{E}#1\bkt{\hspace{1pt}#2\hspace{1.5pt}|\hspace{1.5pt}#3\hspace{1pt}}}
\newcommand{\espcb}[3][]{\mathbb{E}#1\bktb{\hspace{1pt}#2\bigm|#3\hspace{1pt}}}
\newcommand{\espcB}[3][]{\mathbb{E}#1\bktB{#2\Bigm|#3}}
\newcommand{\espcbb}[3][]{\mathbb{E}#1\bktbb{#2\biggm|#3}}

\newcommand{\R}{\ensuremath{\mathbf{R}}\xspace}
\newcommand{\Q}{\ensuremath{\mathbf{Q}}\xspace}
\newcommand{\N}{\ensuremath{\mathbf{N}}\xspace}

\newcommand{\Rpe}{\ensuremath{\R_+^{\,*}}\xspace}
\newcommand{\indi}{\ensuremath{\mathbf{1}}\xspace}
\newcommand{\eps}{\varepsilon}

\newcommand{\card}[1]{\abs{#1}}

\DeclareMathOperator{\cov}{Cov}
\newcommand{\vsp}{\vspace{.15cm}}
\newcommand{\vspar}{\vspace{.25cm}}

\endlocaldefs

\begin{document}

\begin{frontmatter}

\title{An increment type set-indexed Markov property}
\runtitle{An increment type set-indexed Markov property}

\author{\fnms{Paul} \snm{Balan\c{c}a}%
\thanksref{t1}%
\ead[label=e1]{paul.balanca@ecp.fr}%
\ead[label=u1,url]{www.mas.ecp.fr/recherche/equipes/modelisation\_probabiliste}%
}

\address{\printead{e1}\\ \printead{u1}\\[1em] 
\'Ecole Centrale Paris\\
MAS Laboratory\\ 
Grande Voie des Vignes - 92295 Ch\^atenay-Malabry, France
}

\thankstext{t1}{This article is part of the Ph.D. thesis prepared by the author under the supervision of Erick Herbin.}

\affiliation{\'Ecole Centrale Paris}
\runauthor{P. Balan\c{c}a}

\begin{abstract}
We present and study a Markov property, named \emph{\Ci-Markov}, adapted to processes indexed by a general collection of sets. This new definition fulfils one important expectation for a set-indexed Markov property: there exists a natural generalization of the concept of transition operator which leads to characterization and construction theorems of \emph{\Ci-Markov} processes. Several usual Markovian notions, including \emph{Feller} and \emph{strong Markov} properties, are also developed in this framework. Actually, the \emph{\Ci-Markov} property turns out to be a natural extension of the two-parameter \emph{$\ast$-Markov} property to the multiparameter and the set-indexed settings. Moreover, extending a classic result of the real-parameter Markov theory, sample paths of multiparameter \emph{\Ci-Feller} processes are proved to be almost surely right-continuous. Concepts and results presented in this study are illustrated with various examples.
\end{abstract}



\end{frontmatter}

\section{Introduction} \label{sec:introduction}

The Markov property is a central concept of the classic theory of real-parameter stochastic processes. Its extension to processes indexed by a partially ordered collection is a non-trivial problem and multiple attempts exist in the literature to obtain a satisfactory definition. This question has been first investigated by \citet{Levy(1945)} who introduced the \emph{sharp Markov} property for two-parameter processes: a process $(X_t)_{t\in\ivff{0,1}^2}$ is said to be \emph{sharp Markov} with respect to a set $A\subset\ivff{0,1}^2$ if the $\sigma$-fields $\Fi_A$ and $\Fi_{A^c}$ are conditionally independent given $\Fi_{\partial A}$, where for any $V\subset\ivff{0,1}^2$, $\Fi_V\eqdef\sigma(\brc{X_t; t\in V})$. \citet{Russo(1984)} proved that processes with independent increments are \emph{sharp Markov} with respect to any finite unions of rectangles. Later, \citet{Dalang.Walsh(1992),Dalang.Walsh(1992)a} characterized entirely the collection of sets with respect to which processes with independent increments are \emph{sharp Markov}.

Since the Brownian sheet was known not to satisfy this property with respect to simple sets (e.g. triangle with vertices $(0,0)$, $(0,1)$ and $(1,0)$), \citet{McKean(1963)} has introduced a weaker Markov property called \emph{germ-Markov}. Similarly to the sharp Markov property, a process is said to be \emph{germ-Markov} with respect to a set $A\subset\ivff{0,1}^2$ if $\Fi_A$ and $\Fi_{A^c}$ are conditionally independent given $\Gi_{\partial A}\eqdef\cap_{V} \Fi_V$, where the intersection is taken over all open sets $V$ containing $\partial A$. 
As shown by \citet{Russo(1984)}, the Brownian sheet turns out to be \emph{germ-Markov} with respect to any open set.

A third Markov property, named \emph{$\ast$-Markov}, arises in the two-parameter literature. It has been first presented by \citet{Cairoli(1971)}, and then widely studied by \citet{Nualart.Sanz(1979),Korezlioglu.Lefort.ea(1981),Mazziotto(1988)}. In particular, the latter established that \emph{$\ast$-Markov} processes satisfy both \emph{sharp} and \emph{germ-Markov} properties with respect to respectively finite unions of rectangles and convex domains. Moreover, as stated in \cite{Nualart(1983),Luo(1988), Zhou.Zhou(1993)}, the concept of \emph{$\ast$-transition function} can be naturally introduced and leads to a complete characterization of the finite-dimensional distributions of $\R^2$-indexed processes. For the sake of readability, the precise definition of $\ast$-Markov processes is given in the core of the article.

This overview of the two-parameter Markov literature is clearly not exhaustive and a more complete picture is given later, including in particular the modern multiparameter approach presented by \citet{Khoshnevisan(2002)}.

Recently, \citet{Ivanoff.Merzbach(2000)} have introduced a set-indexed formalism which allows to define and study a wider class of processes indexed by a collection of sets. Based on this framework, several extensions of classic real-parameter processes have been investigated, including L\'evy processes \cite{Herbin.Merzbach(2013)}, martingales \cite{Ivanoff.Merzbach(2000)} and fractional Brownian motion \cite{Herbin.Merzbach(2006),Herbin.Merzbach(2009)}. The study of the Markovian aspects of set-indexed processes led to the definition of \emph{sharp Markov} and \emph{Markov} properties in \cite{Ivanoff.Merzbach(2000)a}, and of the \emph{set-Markov} property in \cite{Balan.Ivanoff(2002)}. The first paper mainly focused on the adaptation of Paul L\'evy's ideas to the set-indexed formalism. The latter approach, also called \emph{\Qi-Markov}, introduced a stronger property, leading to the definition of a transition system which characterizes the law of a \emph{set-Markov} process. More precisely, Theorem 1 in \cite{Balan.Ivanoff(2002)} states that, given an initial distribution, if a transition system satisfies a Chapman--Kolmogorov Equation \eqref{eq:chapman_kolmogorov_all} and a supplementary invariance Assumption \eqref{eq:set_markov_assumption}, a corresponding \emph{set-Markov} process can be constructed.\vsp

In the present work, we suggest a different approach, named \emph{\Ci-Markov}, for the introduction of a set-indexed Markov property. Our main goal is to obtain a natural definition of the transition probabilities which leads to satisfactory characterization and construction theorems. As later presented in Section \ref{sec:mpCMarkov}, this new set-indexed Markov property also appears to be a natural extension of some existing multiparameter Markov properties. 

To introduce the definition of \emph{\Ci-Markov} processes, let us first recall the notations used in the set-indexed formalism. Throughout, we consider set-indexed processes $X=\brc{X_A; A\in\Ai}$, where the indexing collection \Ai is made up of compact subsets of a locally compact metric space \Ti. Moreover, \Ai is assumed to satisfy the following conditions:
\begin{enumerate}[\it (i)]
  \item it is closed under arbitrary intersections and $A^\circ \neq A$ for all $A\in\Ai$, $A\neq\Ti$;
  \item $\zset\eqdef\bigcap_{A\in\Ai} A$ is a nonempty set (it plays a role equivalent to $0$ in $\R^N_+$);
  \item there is an increasing sequence $(B_n)_{n\in\N}$ of sets in \Ai such that $\Ti=\cup_{n\in\N} B_n$;
  \item \emph{Shape hypothesis}: for any $A, A_1,\dotsc,A_k\in\Ai$ with $A\subseteq\cup_{i=1}^k A_i$, there exists $i\in\brc{1,\dotsc,k}$ such that $A\subseteq A_i$.
\end{enumerate}
For the sake of readability, we restrict the assumptions on \Ai to those required for the development of the \emph{\Ci-Markov} approach. Supplementary properties such as \emph{separability from above} will be added later when necessary. The complete definition of an indexing collection can be found in the work of~\citet{Ivanoff.Merzbach(2000)}. 

Several examples of indexing collection have been presented in the literature (see e.g.~\cite{Ivanoff.Merzbach(2000),Balan.Ivanoff(2002),Herbin.Merzbach(2006)}). Among them, let us mention:
\begin{itemize}
  \item rectangles of $\R^N_+$: $\Ai=\brc{\ivff{0,t} : t\in\R_+^N}$. This collection is equivalent to the usual multiparameter setting on $\R^N_+$;
  \item subsets of the $N$-dimensional unit sphere $\mathcal{S}_N$: $\Ai=\brc{A_{\varphi} : \varphi\in\ivff{0,\pi}^{N-1}\times\ivfo{0,2\pi}}$, where $A_{\varphi}\eqdef\brc{u\in\mathcal{S}_N : \phi_i(u) \leq \varphi_i, \ 1\leq i\leq N }$, $\phi_i(u)$ denoting the $i$th angular coordinates of $u$, $1\leq i\leq N$;
  \item branches of a tree $T$ parametrized by $G_T\subset \cup_{n\in\N} (\N^*)^n$: $\Ai = G_T$, endowed with the usual intersection on tree structures. Note that \Ai is a discrete indexing collection.
\end{itemize}
More complex indexing collections can be obtained by considering the Cartesian product of simpler ones. This construction procedure is described more thoroughly in Section~\ref{ssec:si_examples}.\vsp

The notations \Aiu and \Ci respectively refer to the class of finite unions of sets from \Ai and to the collection of increments $C = A\setminus B$, where $A\in\Ai$ and $B\in\Aiu$. As mentioned in \cite{Ivanoff.Merzbach(2000)}, the assumption \emph{Shape} on \Ai implies the existence of a unique \emph{extremal representation} $\brc{A_i}_{1\leq i\leq k}$ of every $B\in\Aiu$, i.e. such that $B = \cup_{i=1}^k A_i$ and for every $i\neq j$, $A_i\nsubseteq A_j$. Hence, from now on, to any increment $C\in\Ci$ is associated the unique $A\in\Ai$ and $B = \cup_{i=1}^k A_i$ such that $C=A\setminus B$, $B\subseteq A$ and $\brc{A_i}_{i\leq k}$ is the extremal representation of $B$.

Then, for any $C = A\setminus B$, $B = \cup_{i=1}^k A_i$, \Cp{C} denotes the following subset of \Ai,
\begin{equation}\label{eq:def_Cfrontier}
  \Cp{C} = \brc{U\in\Ai_\ell;\, U\nsubseteq B^\circ} \eqdef \brc{\Cpi{C}{1},\dotsb,\Cpi{C}{p}}, \qquad\text{where }p=\card{\Cp{C}}
\end{equation}
and $\Ai_\ell$ denotes the semilattice $\brc{A_1\cap\dotsb\cap A_k,\dotsc,A_1\cap A_2,A_1\dotsc,A_k}\subset\Ai$.
For any set-indexed stochastic process $X$ and any $C\in\Ci$, the notation \vXCp{C} refers to the random vector $\vXCp{C} = \pthb{ X_{\Cpi{C}{1}},\dotsc,X_{\Cpi{C}{p} } }$. Similarly, \vxCp{C} is used to denote a vector of variables $\pth{ x_{\Cpi{C}{1}},\dotsc,x_{\Cpi{C}{p} } }$. 

Lastly, for any set-indexed filtration $\Fi=\brc{ \Fi_A; A\in\Ai}$, collections $(\Fi_B)_{B\in\Aiu}$ and $(\Gc{C})_{C\in\Ci}$ are respectively defined as
\begin{equation} \label{eq:def_filtrations}
  \Fi_B \eqdef \bigvee_{A\in\Ai,A\subseteq B}\,\Fi_A \qquad\text{and}\qquad \Gc{C} \eqdef \bigvee_{B\in\Aiu,B\cap C=\vset}\,\Fi_B.
\end{equation}
The family $\pth{ \Gc{C} }_{C\in\Ci}$ is usually called the \emph{strong history}. Note that these filtrations do not necessarily satisfy any outer-continuity property, on the contrary to the augmented filtrations defined later.

We can now present the set-indexed increment Markov property, abbreviated \emph{\Ci-Markov} property (\Ci denoting the class of increment sets) which is studied in the present work.
\begin{definition}[\textbf{\Ci-Markov property}] \label{def:c_markov}
  Let $(\Omega,\Fi,(\Fi_A)_{A\in\Ai},\Pr)$ be a complete probability space and $E$ be a metric space. An $E$-valued set-indexed process $X$ is said to be \emph{\Ci-Markov} with respect to (w.r.t.) the filtration $(\Fi_A)_{A\in\Ai}$ if it is adapted to $(\Fi_A)_{A\in\Ai}$ and if it satisfies
  \begin{equation}  \label{eq:c_markov}
    \espc{f(X_A)}{\Gc{C}} = \espc{f(X_A)}{\vXCp{C}} \quad\text{\Pr-a.s.}
  \end{equation}
  for all $C = A\setminus B\in\Ci$ and any bounded measurable function $f:E\rightarrow \R$.
\end{definition}

\Ci-Markov processes are first studied in the general set-indexed framework (Section \ref{sec:siCMarkov}). We introduce the concept of \emph{\Ci-transition system}, usually designated $\Pii=\brc{\Pc{C}(\vxCp{C};\dt x_A) ; C\in\Ci}$, associated to the \Ci-Markov property. It extends in a natural way the classic definition of transition operators for one-parameter processes, and in particular, the Chapman--Kolmogorov equation which becomes:
\begin{equation} \label{eq:chapman_kolmogorov_short}
  \forall C\in\Ci,\, A'\in\Ai;\qquad \Pc{C} f = \Pc{C'} \Pc{C''} f\quad\text{where }\ C' = C\cap A'\text{, } C''=C\setminus A' 
\end{equation}
and $f$ is a bounded measurable function. This notion happens to properly suit the \Ci-Markov property. Indeed, our main result (Theorems \ref{th:c_markov_charac} and \ref{th:c_markov_construct} in Section \ref{ssec:si_def_construct}) states that the initial law of $X_\zset$ and the transition probabilities naturally deduced from  a \Ci-Markov process form a \Ci-transition system which characterizes entirely the finite-dimensional distributions. Moreover, given an initial measure $\nu$ and \Ci-transition system \Pii, we state that a corresponding \Ci-Markov process can be constructed on the canonical space.

As presented in Section \ref{ssec:si_markov}, the \Ci-Markov property has connections with the set-indexed Markovian literature, especially the work of \citet{Ivanoff.Merzbach(2000)a}. However, we note there that it clearly differs from the \emph{\Qi-Markov} approach of \citet{Balan.Ivanoff(2002)}.

\Ci-Markov processes appear to satisfy several interesting properties (Sections \ref{ssec:si_properties} and \ref{ssec:si_feller}). Theorem \ref{th:c_markov_CI} proves that the natural filtration of \Ci-Markov process displays a \emph{conditional independence} property. Moreover, classic \emph{simple} and  \emph{strong Markov} properties can be extended to the \Ci-Markov formalism (Theorems \ref{th:c_markov_simple} and \ref{th:c_markov_strong}) under an homogeneity assumption on the \Ci-transition system, and for the latter, with the help of the notion of \Ci-Feller processes. To illustrate these different concepts, several examples are given in Section \ref{ssec:si_examples}, including an overview of set-indexed L\'evy processes and the construction of a set-indexed \SaS Ornstein--Uhlenbeck process.

Finally, Section \ref{sec:mpCMarkov} is devoted to the specific properties satisfied by multiparameter \Ci-Markov processes. Specifically, Theorem \ref{th:mpcad} states that multiparameter \Ci-Feller processes have a right-continuous modification which is \Ci-Markov with respect to the augmented filtration. Connections with the multiparameter Markovian literature are also investigated, especially with the two-parameter $\ast$-Markov property and the recent multiparameter developments presented by \citet{Khoshnevisan(2002)}. Several examples of common multiparameter \Ci-Markov processes are given at the end of the section.

\section{\texorpdfstring{Set-indexed \Ci-Markov property}{Set-indexed C-Markov property}} \label{sec:siCMarkov}

We begin this section with a few observations and remarks. 

In this work are considered $E$-valued set-indexed processes $X=\brc{X_A; A\in\Ai}$, where $(E,d_E)$ is a locally compact separable metric space endowed with the Borel sigma-algebra $\Ei$. If $(E,\Ei) = (\R^d,\Bi(\R^d))$, then the assumption \emph{shape} on the indexing collection allows to define the extension $\Delta X$ of $X$ on the classes \Aiu and \Ci. It is given by the following inclusion-exclusion formulas:
\[
  \Delta X_B \eqdef \sum_{i=1}^k X_{A_i} - \sum_{i<j} X_{A_i \cap A_j} + \dotsb + (-1)^{k+1} X_{A_1\cap\dotsb\cap A_k} \quad\text{and}\quad \Delta X_C \eqdef X_A - \Delta X_B.
\]
Due to Lemma 3.4 by \citet{Ivanoff.Merzbach(2000)a} and using the notations previously introduced, we observe that the previous formulas can be equivalently written as follows:
\begin{align} \label{eq:inc_exc_formula}
  \Delta X_B = \sum_{i=1}^{\abs{\Cp{C}}} (-1)^{\eCpi{i}} X_{\Cpi{C}{i}} \qquad\text{and}\qquad \Delta X_C = X_A - \bktbb{ \sum_{i=1}^{\abs{\Cp{C}}} (-1)^{\eCpi{i}} X_{\Cpi{C}{i}} },
\end{align}
where $(-1)^{\eCpi{i}}$ corresponds the sign in front of $X_{\Cpi{C}{i}}$ in the inclusion-exclusion formula. 
In other words, Equation \eqref{eq:inc_exc_formula} states that every term $X_V$ such that $V\notin\Cp{C}$ is cancelled by another element in the inclusion-exclusion formula. 

In consequence, we note that Definition \ref{def:c_markov} can be equivalently written as follows: an $\R^d$-valued set-indexed process $X$ is \Ci-Markov with respect to $(\Fi_A)_{A\in\Ai}$ if for all $C\in\Ci$ and any measurable function $f:E\rightarrow\R_+$,
\[
  \espc{f(\Delta X_C)}{\Gc{C}} = \espc{f(\Delta X_C)}{\vXCp{C}} \quad\text{\Pr-a.s.}
\]
Indeed, owing to Equation \eqref{eq:inc_exc_formula}, $\Delta X_B$ is measurable with respect to $\Gc{C}$ and $\sigma(\vXCp{C})$, and therefore, the equality $X_A = \Delta X_C + \Delta X_B$ and classic properties of the conditional expectation imply the equivalence of the two definitions.

Finally, let us recall that the natural filtration of a set-indexed process $X$ is defined by $\Fi_A = \sigma(\brc{X_V; V\subseteq A, V\in\Ai})$ for all $A\in\Ai$. Filtrations are always supposed to be complete. Based on Definition \ref{def:c_markov}, we observe that any \Ci-Markov process is always \Ci-Markov with respect to its natural filtration.

\subsection{\texorpdfstring{\Ci-transition system: characterization and construction results}{C-transition system, characterization and construction}} \label{ssec:si_def_construct}

In the light of Definition \ref{def:c_markov} of \Ci-Markov processes, there is a natural way to introduced the concept of \emph{\Ci-transition system}.
\begin{definition}[\textbf{\Ci-transition system}] \label{def:c_markov_tr}
  A collection $\Pii = \brc{\Pc{C}(\vxCp{C};\dt x_A);\,C\in\Ci}$ is called a \emph{\Ci-transition system} if it satisfies the following properties:
  \begin{enumerate}[ \it 1.]
    \item for all $C\in\Ci$, $\Pc{C}(\vxCp{C};\dt x_A)$ is a transition probability, i.e. for all $\vxCp{C}\in E^{\abs{\Cp{C}}}$, $\Pc{C}(\vxCp{C};\,\cdot\,)$ is a probability measure on $(E,\Ei)$ and for all $\Gamma\in\Ei$, $\Pc{C}(\,\cdot\,;\Gamma)$ is a measurable function;\vsp
    \item for all $x\in E$ and $\Gamma\in\Ei$, $\Pc{\vset}(x;\Gamma) = \delta_x(\Gamma)$; \vsp
    \item for all $C\in\Ci$ and any $A'\in\Ai$, let $C' = C\cap A'$ and $C'' = C\setminus A'$. Then, \Pii satisfies a Chapman--Kolmogorov like equation:
    \begin{equation} \label{eq:chapman_kolmogorov}
      \Pc{C}(\vxCp{C};\Gamma) = \int_E \Pc{C'}(\vxCp{C'};\dt x_{A'}) \, \Pc{C''}(\vxCp{C''};\Gamma),
    \end{equation}
    for all $\vxCp{C}\in E^{\abs{\Cp{C}}}$, $\vxCp{C'}\in E^{\abs{\Cp{C'}}}$, $\vxCp{C''}\in E^{\abs{\Cp{C''}}}$ and $\Gamma\in\Ei$.
  \end{enumerate}
\end{definition}

\begin{remark}
  Let $C = A\setminus B\in\Ci$ and $A'\in\Ai$. We observe that if $A'\subseteq B$ or $A\subseteq A'$, then $C'=\vset$ or $C''=\vset$ respectively, and therefore, Equation \eqref{eq:chapman_kolmogorov} is straight forward since one of the term is a Dirac distribution.

  On the other hand, if $A'\nsubseteq B$ and $A'\subset A$, Equation \eqref{eq:chapman_kolmogorov} is still consistent. Indeed, as $C'' = A\setminus(A'\cup B)$, and $A'\nsubseteq B$, we have $A'\in\Cp{C''}$, and thus, the variable $x_{A'}$ is one of the component of the vector $\vxCp{C''}$.

  We also note that Equation \eqref{eq:chapman_kolmogorov} implicitly induces that the integral does not depend on any variable $x_V$ with $V\notin\Cp{C}$, since these terms do not appear in the left-term. 
\end{remark}

\begin{remark}
  In the particular case of $\Ti=\R_+$ and $\Ai=\brc{\ivff{0,t};\, t\in\R_+}$, Definition \ref{def:c_markov_tr} corresponds to the usual definition of a transition system. Indeed, as $\Aiu = \Ai$ and $\Ci = \brc{\ivof{s,t};\, s,t\in\R_+^2}$, $\Pii$ is indexed by $\R_+^2$: $\Pii = \brc{\Pc{s,t}(x;\dt y);\, (s,t)\in\R_+^2}$. Then, Definition \ref{def:c_markov_tr} states that:
  \begin{enumerate}[ \it 1.]
    \item for all $s,t\in\R_+^2$, $\Pc{s,t}(x;\dt y)$ is a transition probability;
    \item for all $s\in\R_+$, $\Pc{s,s}(x,\dt y) = \delta_x(\dt y)$;
    \item for all $s<s'<t\in\R_+$, $x\in E$ and $\Gamma\in\Ei$, \Pii satisfies
    $\Pc{s,t}(x;\Gamma) = \int_E \Pc{s,s'}(x;\dt y)\;\Pc{s',t}(y;\Gamma)$ (using previous notations, $C = \ivof{s,t}$, $C'=\ivof{s',t}$ and $C''=\ivof{s,s'}$).
  \end{enumerate}
\end{remark}

Proposition \ref{prop:c_markov_proc_tr} proves that Definition \ref{def:c_markov_tr} of a \Ci-transition systems is coherent with the \Ci-Markov property.
\begin{proposition} \label{prop:c_markov_proc_tr}
  Let $(\Omega,\Fi,(\Fi_A)_{A\in\Ai},\Pr)$ be a complete probability space and $X$ be a $\Ci$-Markov process w.r.t. $(\Fi_A)_{A\in\Ai}$. For all $C=A\setminus B\in\Ci$, define $\Pc{C}(\vxCp{C};\dt y)$ as follows:
  \[
    \forall \vxCp{C}\in E^{\abs{\Cp{C}}}, \Gamma\in\Ei;\quad \Pc{C}(\vxCp{C};\Gamma) \eqdef \prc{X_A\in\Gamma}{\vXCp{C}=\vxCp{C}}.
  \]
  Then, the collection $\Pii=\brc{\Pc{C}(\vxCp{C};\dt y);\,C\in\Ci}$ is a \Ci-transition system.
\end{proposition}
\begin{proof}
  Let us verify the different points of Definition \ref{def:c_markov_tr}.
  \begin{enumerate}[ \it 1.]
    \item Clearly, for any $C=A\setminus B\in\Ci$, $\Pc{C}(\vxCp{C};\dt x_A)$ is a transition probability.\vsp
    
    \item The equality $\Pc{\vset}(x;\dt y) = \delta_x(\dt y)$ holds since, $\Gc{\vset}=\vee_{V\in\Ai} \Fi_V$ and therefore, for any $f$ bounded measurable function and any $A\in\Ai$, $\espc{f(X_A)}{\Gc{\vset}} = f(X_A)$.\vsp
    
    \item For all $C = A\setminus B\in\Ci$ and all $A'\in\Ai$ such that $A'\subseteq A$ and $A'\nsubseteq B$, let $C' = C\cap A' = A'\setminus (A'\cap B)$ and $C'' = C\setminus A' = A\setminus(A'\cup B)$. Since $C''\subset C$, we observe that $\Gc{C}\subseteq\Gc{C''}$ and thus, for any $\Gamma\in\Ei$,
    \begin{align} \label{eq:dem1}
      \Pc{C}(\vXCp{C};\Gamma) = \prc{X_A\in\Gamma}{\Gc{C}} = \prcb{\prc{X_A\in\Gamma}{\Gc{C''}} }{\Gc{C}} = \espc{ \Pc{C''}(\vXCp{C''};\Gamma) }{\Gc{C}}.
    \end{align}
    Furthermore, $\Gc{C}\subseteq\Gc{C'}$ and the vector $\vXCp{C'}$ is $\Gc{C}$-measurable ($A'\cap B\subset B$). Hence, for any positive measurable function $h$,
    \begin{equation} \label{eq:dem2}
      \espc{h(X_{A'})}{\Gc{C}} = \espcb{ \espc{h(X_{A'})}{\Gc{C'}} }{\Gc{C}} = \int_\R \Pc{C'}(\vXCp{C'};\dt x_{A'}) \, h(x_{A'}).
    \end{equation}
    As $C'' = A\setminus(A'\cup B)$, we observe that $X_{A'}$ is the only term in the vector $\vXCp{C''}$ which is not \Gc{C}-measurable.
    Therefore, using a monotone class argument, Equations \eqref{eq:dem1} and \eqref{eq:dem2} lead to the expected equality.
  \end{enumerate}
\end{proof}
The next two theorems gather the main result of this section: for any \Ci-Markov process, its initial distribution and its \Ci-transition system characterize entirely the law of the process. Conversely, from any probability measure and any \Ci-transition system, a corresponding canonical \Ci-Markov process can be constructed. To our knowledge, such a result does not exist for other set-indexed Markov properties, or at least require some tricky technical assumption in the case of \Qi-Markov.

\begin{theorem} \label{th:c_markov_charac}
  Let $(\Omega,\Fi,(\Fi_A)_{A\in\Ai},\Pr)$ be a complete probability space and $X$ be a $\Ci$-Markov process w.r.t. $(\Fi_A)_{A\in\Ai}$. 
  
  Then, the initial distribution $\nu$ (i.e. the law of $X_\zset$) and the \Ci-transition system $\Pii = \brc{\Pc{C}(\vxCp{C};\dt x_A);\,C\in\Ci}$ of the process $X$ characterize entirely its law.
\end{theorem}
\begin{proof}
  We have to express the finite-dimensional distributions of $X$ in terms of $\nu$ and $\Pii$.
  
  Let $A_1,\dotsc,A_k\in\Ai$ and $f: E^k\rightarrow \R_+$ be a measurable function. Without any loss of generality, we can assume that $\Ai_\ell=\brc{A_0 = \zset,A_1,\dotsc,A_k}$ is a finite semilattice with consistent numbering, i.e. stable under intersection and such that $A_j\subseteq A_i$ implies $j\leq i$. 
  
  Let $C_1,\dots,C_k\in\Ci$ be the left-neighbourhoods $C_i = A_i \setminus (\cup_{j=0}^{j-1} A_j)$, $1\leq i\leq k$. Since $\Ai_\ell$ is semilattice, $\Cp{C_i}\subset \Ai_\ell$ for any $i\in\brc{1,\dotsc,k}$.
  
  Then, if we consider the law of $(X_{A_1},\dotsc,X_{A_k})$, we observe that
  \begin{align*}
    \esp{f(X_{A_1},\dotsc,X_{A_k})}
    &= \espb{ \espc{ f(X_{A_1},\dotsc,X_{A_k}) }{\Gc{C_k}} } \\
    &= \espB{ \int_E \Pc{C_k}(\vXCp{C_k};\dt x_{A_k}) f(X_{A_1},\dotsc,X_{A_{k-1}},x_{A_k} ) },
  \end{align*}
  using a monotone class argument and the $\Gc{C_{k}}$-measurability of the vector $(X_{A_1},\dotsc,X_{A_{k-1}})$. Therefore, by induction,
  \begin{align*}
    &\esp{f(X_{A_1},\dotsc,X_{A_k})} \\
    &= \espB{ \int_{E^k} \Pc{C_1}(X_\zset;\dt x_{A_1})\, \Pc{C_2}(\vXCp{C_2};\dt x_{A_2})\dotsb \Pc{C_k}(\vXCp{C_k};\dt x_{A_k})\, f(x_{A_1},\dots,x_{A_k}) } \\
    &= \int_{E^{k+1}} \nu(\dt x_0)  \Pc{C_1}(x_0;\dt x_{A_1})\, \Pc{C_2}(\vXCp{C_2};\dt x_{A_2})\dotsb \Pc{C_k}(\vXCp{C_k};\dt x_{A_k})\, f(x_{A_1},\dots,x_{A_k}),
  \end{align*}
  since $\Cp{C_1} = \brc{\zset}$.
\end{proof}

To obtain the construction theorem, we need the following technical lemma.
\begin{lemma} \label{lemma:markov_semilattice}
  Consider a \Ci-transition system $\Pii = \brc{\Pc{C}(\vxCp{C};\dt x_A);\,C\in\Ci}$ and two semilattices $\Ai^1 = \brc{A^1_0=\zset,A^1_1,\dotsc,A^1_{n_1}}$ and $\Ai^2 = \brc{A^2_0=\zset,A^2_1,\dotsc,A^2_{n_2}}$ consistently numbered and such that $\Ai^1\subseteq\Ai^2\subset\Ai$. 
  
  Let $C^1_1,\dotsc,C^1_{n_1}\in\Ci$ and $C^2_1,\dotsc,C^2_{n_2}\in\Ci$ be the left-neighbourhoods of $\Ai^1$ and $\Ai^2$, respectively, i.e. $C^1_i = A^1_i\setminus(\cup_{j=0}^{i-1}A^1_j)$ and $C^2_i = A^2_i\setminus(\cup_{j=0}^{i-1}A^2_j)$. Then, for all $x\in E$ and any positive measurable function $f:E^{n_1}\mapsto\R_+$, \Pii satisfies
  \begin{align} \label{eq:big_lemma}
     & \int_{E^{n_1}} \Pc{C^1_1}(x;\dt x_{A^1_1})\,\Pc{C^1_2}(\vxCp{C^1_2};\dt x_{A^1_2})\dotsb \Pc{C^1_{n_1}}(\vxCp{C^1_{n_1}};\dt x_{A^1_{n_1}})\, f(x_{A^1_1},\dotsc,x_{A^1_{n_1}}) \nonumber \\
    =& \int_{E^{n_2}} \Pc{C^2_1}(x;\dt x_{A^2_1})\,\Pc{C^2_2}(\vxCp{C^2_2};\dt x_{A^2_2})\dotsb \Pc{C^2_{n_2}}(\vxCp{C^2_{n_2}};\dt x_{A^2_{n_2}})\, f(x_{A^1_1},\dotsc,x_{A^1_{n_1}}).
  \end{align}
  We note that since $\Ai^1\subseteq\Ai^2$, every variable $x_{A^1_i}$ has a corresponding term $x_{A^2_{j_i}}$ where $j_i\in\brc{1,\dotsc,n_2}$, ensuring the consistency of the second term in Equation \eqref{eq:big_lemma}.
\end{lemma}
\begin{proof}
  We proceed by induction on $n_2$. Let $k\in\brc{1,\dotsc,n_2}$ be the highest index such that $A^2_k\in\Ai^2$ and $A^2_k\notin\Ai^1$. Let ${\Ai^2}'$ be $\Ai^2 \setminus \brc{A^2_k}$. We observe that ${\Ai^2}'$ is semilattice, since otherwise, we would have $A^2_k\in\Ai^1$. We consider two different cases.
  \begin{enumerate}[ \it 1.]
    \item Suppose first that for every $i\in\brc{k+1,\dotsc,n_2}$, $A^2_k\notin \Cp{C_i^2}$. 
    Then, we know that for any $i\in\brc{1,\dotsc,n_2}\setminus\brc{k}$, the term $x_{A^2_k}$ does not appear in the component $\Pc{C_i}(\vxCp{C^2_i};\dt x_{A^2_i})$.
    Hence, we can integrate over the variable $x_{A^2_k}$, and since $\int_E \Pc{C^2_k}(\vxCp{C^2_k};\dt x_{A^2_k}) = 1$, we obtain a formula which corresponds to the case of the semilattice ${\Ai^2}'$. \vsp
    
    \item Assume now that there exist $i_m>\dotsb>i_1>k$ such that $A^2_k\in\Cp{C_{i_j}^2}$ for every $j\in\brc{1,\dotsc,m}$. Let us first suppose that $m\geq 2$. 
    
    Then, there exists $i_0\in\brc{1,\dotsc,i_1}$ such that $A^2_{i_0} = A^2_{i_1}\cap A^2_{i_2}$. As $A^2_k\subseteq A^2_{i_1}\cap A^2_{i_2}$, we have $k\leq i_0$. Furthermore, $A^2_{i_0}= A^2_k$ is not possible since $A^2_{i_1},A^2_{i_2}$, and thus $A^2_{i_0}$, belong to $\Ai^1$. Thereby, we necessarily have $A^2_k\in\Cp{C_{i_0}^2}$, which implies $i_0=i_1$ because of $i_1$'s definition.
  
    Hence, $A^2_{i_1}\subset A^2_{i_2}$. But, since $A^2_k\subset A^2_{i_1}$, we have
    \[
      C^2_{i_2} \eqdef A^2_{i_2} \setminus\pth{\cup_{j=0}^{i_2-1} A^2_j} = A^2_{i_2} \setminus\pth{\cup_{j=0,j\neq k}^{i_2-1} A^2_j}.
    \]
    Furthermore, we note that $\brc{A^2_1,\dotsc,A^2_{i_2-1}}\setminus\brc{A^2_k}$ is a semilattice (otherwise we would have $A^2_k\in\Ai^1$), therefore stable under intersections. Hence, due to the previous observations, $\Cp{C^2_{i_2}} \subseteq \brcb{A^2_j\,;\, 0\leq j < i_2 \text{ and } j\neq k}$, which is in contradiction with the assumption $A^2_k\in\Cp{C^2_{i_2}}$. 
     
    Therefore, $m=1$ and without any loss of generality, we can assume that $i_1 = k+1$. The variable $x_{A^2_k}$ is only present in the vector $\vxCp{C^2_{k+1}}$ and thus, our problem is reduced to the computation of the sub-integral $\int_E \Pc{C^2_k}(\vxCp{C^2_k};\dt x_{A^2_k})\, \Pc{C^2_{k+1}}(\vxCp{C^2_{k+1}};\Gamma)$, where $\Gamma\in\Ei$.
    
    Let $C\in\Ci$ be $C=A^2_{k+1}\setminus(\cup_{j=0}^{k-1} A^2_j)$. We observe that $C^2_k = C\cap A^2_k$ and $C^2_{k+1} = C \setminus A^2_k$. Therefore, owing to the Chapman--Kolmogorov Equation \eqref{eq:chapman_kolmogorov}, 
    \[
      \Pc{C}(\vxCp{C};\Gamma) = \int_E \Pc{C^2_k}(\vxCp{C^2_k};\dt x_{A^2_k})\, \Pc{C^2_{k+1}}(\vxCp{C^2_{k+1}};\Gamma).
    \]
    Finally, if we insert the previous formula in Equation \eqref{eq:big_lemma}, we obtain an integral which also corresponds to the case of the semilattice ${\Ai^2}'$.
  \end{enumerate}
  In both cases, the right-integral is simplified into a formula corresponding to the case of the semilattice ${\Ai^2}' = \Ai^2 \setminus \brc{A^2_k}$. Hence, by induction on the size of $\Ai^2$, we obtain Equation~\eqref{eq:big_lemma}.
\end{proof}

We are now able to prove that a canonical \Ci-Markov process can be constructed from any initial probability measure and any \Ci-transition system.

In the following proposition, $\Omega$ designates the canonical space $E^{\Ai}$ endowed with its usual $\sigma$-field \Fi generated by coordinate applications. $X$ denotes the canonical process, i.e. for all $A\in\Ai$ and all $\omega\in\Omega$, $X_A(\omega)=\omega(A)$ and $(\Fi_A)_{A\in\Ai}$ is its natural filtration (as defined at the beginning of this section).

\begin{theorem} \label{th:c_markov_construct}
  Let $\Pii = \brc{\Pc{C}(\vxCp{C};\dt x_A);\,C\in\Ci}$ be a \Ci-transition system and $\nu$ be a probability measure $\nu$ on $(E,\Ei)$.
  
  Then, there exists a unique probability measure $\Pr_\nu$ on $(\Omega,\Fi)$ such that $X$ is a \Ci-Markov process w.r.t. $(\Fi_A)_{A\in\Ai}$ and whose initial measure and \Ci-transition system are $\nu$ and \Pii, respectively:
  \[
    X_\zset \sim \nu\qquad\text{and}\qquad\prc[_\nu]{X_A\in\Gamma}{\Gc{C}} = \Pc{C}(\vXCp{C};\Gamma) \quad\text{$\Pr_\nu$-a.s.}
  \]
  for all $C = A\setminus B\in\Ci$ and all $\Gamma\in\Ei$.
\end{theorem}
\begin{proof}
  To construct the measure $\Pr_\nu$, we need to define for any $A_1,\dotsc,A_k\in\Ai$ a probability measure $\mu_{A_1 \dotsc A_k}$ on $(E^k,\Ei^{\otimes k})$ such that this family of probabilities satisfies the usual consistency conditions of Kolmogorov's extension theorem.
  \begin{enumerate}[ \it 1.]
    \item Let $A_1,\dotsc,A_k$ be $k$ distinct sets in $\Ai$ and $\pi$ be a permutation on $\brc{1,\dotsc,k}$. Then,
      \begin{equation} \label{eq:kolmo_consistency1}
        \mu_{A_1\dotsc A_k}(\Gamma_1 \times\dotsb\times \Gamma_k) = \mu_{A_{\pi(1)}\dotsc A_{\pi(k)}}(\Gamma_{\pi(1)} \times\dotsb\times \Gamma_{\pi(k)} )
      \end{equation}
      for all $\Gamma_1,\dotsc,\Gamma_k\in\Ei$.
    \item Let $A_1,\dotsc,A_k,A_{k+1}$ be $k+1$ distinct sets in \Ai. Then,
      \begin{equation} \label{eq:kolmo_consistency2}
        \mu_{A_1\dotsc A_k}(\Gamma_1 \times\dotsb\times \Gamma_k) = \mu_{A_1\dotsc A_k A_{k+1}}(\Gamma_1 \times\dotsb\times \Gamma_k \times E)
      \end{equation}
      for all $\Gamma_1,\dotsc,\Gamma_k\in\Ei$.
  \end{enumerate}
  For any $A_1,\dotsc,A_k$ distinct sets in $\Ai$, we consider $\Ai'=\brc{A'_0=\zset,A'_1,\dotsc,A'_m}$ the smallest semilattice generated by $\brc{\zset,A_1,\dotsc,A_k}$. Let $C_1,\dots,C_m\in\Ci$ be the left-neighbourhoods on $\Ai'$: $C_i = A'_i\setminus(\cup_{j=0}^{i-1} A'_j)$. Then, we define the probability measure $\mu_{A_1\dotsc A_k}$ on $(E^k,\Ei^{\otimes k})$ as follows:
  \begin{align*}
    &\mu_{A_1\dotsc A_k}(\Gamma_1 \times\dotsb\times \Gamma_k) \\
    &= \int_{E^{m+1}} \nu(\dt x_0)\, \Pc{C_1}(x_0;\dt x_{A'_1})\,\Pc{C_2}(\vxCp{C_2};\dt x_{A'_2})\dotsb \Pc{C_m}(\vxCp{C_m};\dt x_{A'_m}) \indi_{\Gamma_1}(x_{A_1})\dotsb\indi_{\Gamma_k}(x_{A_k}).
  \end{align*}
  where $\Gamma_1,\dotsc,\Gamma_k\in\Ei$. Note that for every variable $x_{A_i}$, $i\in\brc{1,\dotsc,k}$, there exists a corresponding $x_{A'_{j_i}}$, $j_i\in\brc{1,\dotsc,m}$. Moreover, the formula above does not depend on the numbering of $\Ai'$, since neither $C_i$ nor $\vxCp{C_i}$, $i\in\brc{1,\dotsc,m}$, are changed with a different consistent ordering. These two remarks ensure that the previous definition is consistent. \vsp
  
  Consider now the first consistency condition that must satisfy $\mu_{A_1\dotsc A_k}$. Let $A_1,\dotsc,A_k$ be $k$ distinct sets in $\Ai$, $\Gamma_1,\dotsc,\Gamma_k$ be in \Ei and $\pi$ be a permutation on $\brc{1,\dotsc,k}$. Then, we know that families $\brc{A_1,\dots,A_k}$ and $\brc{A_{\pi(1)},\dots,A_{\pi(k)}}$ lead to the same semilattice $\Ai'=\brc{A'_0=\zset,A'_1,\dotsc,A'_m}$. Hence, since the definition of $\mu_{A_1\dotsc A_k}$ does not depend on the ordering of the semilattice $\Ai'$, Equation \eqref{eq:kolmo_consistency1} is verified.
  
  In order to check the second consistency condition, let $A_1,\dotsc,A_{k+1}$ be $k+1$ distinct sets in $\Ai$ and $\Gamma_1,\dotsc,\Gamma_k$ be in $\Ei$. Let $\Ai^1 = \brc{A^1_0=\zset,A^1_1,\dotsc,A^1_{n_1}}$ and $\Ai^2 = \brc{A^2_0=\zset,A^2_1,\dotsc,A^2_{n_2}}$ be the two semilattices generated by $\brc{A_1,\dotsc,A_k}$ and $\brc{A_1,\dots,A_k,A_{k+1}}$, respectively. We clearly have $\Ai^1 \subseteq \Ai^2$. Let finally $f:E^{n_1}\mapsto\R_+$ be a positive measurable function defined as follows:
  \[
    \forall (x_{A^1_1},\dotsc,x_{A^1_{n_1}})\in E^{n_1};\quad f(x_{A^1_1},\dotsc,x_{A^1_{n_1}}) = \indi_{\Gamma_1}(x_{A_1})\dotsm\indi_{\Gamma_k}(x_{A_k}).
  \]
  The definition of $f$ is consistent since $\brc{A_1,\dotsc,A_k}\subseteq\Ai^1$. Furthermore, if we apply Lemma \ref{lemma:markov_semilattice} to $\Ai^1$ and $\Ai^2$, we obtain
  \begin{align*}
    &\int_{E^{n_1+1}} \nu(\dt x_0)\,\Pc{C^1_1}(\vxCp{C^1_{0}};\dt x_{A^1_1})\dotsb \Pc{C^1_{n_1}}(\vxCp{C^1_{n_1}};\dt x_{A^1_{n_1}})\, \indi_{\Gamma_1}(x_{A_1})\dotsb\indi_{\Gamma_k}(x_{A_k}) \\
    = &\int_{E^{n_2+1}} \nu(\dt x_0)\,\Pc{C^2_1}(\vxCp{C^2_{0}};\dt x_{A^2_1}) \dotsb \Pc{C^2_{n_2}}(\vxCp{C^2_{n_2}};\dt x_{A^2_{n_2}})\, \indi_{\Gamma_1}(x_{A_1})\dotsb\indi_{\Gamma_k}(x_{A_k}) \,\indi_{\R}(x_{A_{k+1}}),
  \end{align*}
  which exactly corresponds to the second consistency Equation \eqref{eq:kolmo_consistency2}.
  
  Therefore, using Kolmogorov's extension theorem (see e.g. Appendix A in \cite{Khoshnevisan(2002)}), there exists a probability measure $\Pr_\nu$ on $(\Omega,\Fi)$ such that for all $A_1,\dotsc,A_k$ in \Ai and $\Gamma_1,\dotsc,\Gamma_k$ in \Ei,
  \[
    \pr[_\nu]{X_{A_1}\in\Gamma_1,\dotsc,X_{A_k}\in\Gamma_k} = \mu_{A_1\dotsc A_k}(\Gamma_1 \times\dotsb\times \Gamma_k),
  \]
  where $X$ designates the canonical process. A usual monotone class argument ensures the uniqueness of this probability measure.\vsp
   
  Our last point in this proof is to ensure that under $\Pr_\nu$, $X$ is a \Ci-Markov process w.r.t. $(\Fi_A)_{A\in\Ai}$ whose initial measure and \Ci-transition system are $\nu$ and \Pii.
  The first point is clear, since for any $\Gamma_0\in\Ei$, $\pr[_\nu]{X_\zset\in\Gamma_0} = \nu(\Gamma_0)$.
   
  Consider now $C = A\setminus B$ in \Ci, where $B=\cup_{i=1}^k A_i$. Let $A'_1,\dots,A'_m$ be in \Ai such that $A'_i\cap C=\vset$ for every $i\in\brc{1,\dotsc,m}$. Without any loss of generality, we may assume that $\Ai' = \brc{A'_0=\zset,A'_1,\dotsc,A'_m,A}$ is a semilattice with consistent numbering such that $\brc{A_i}_{i\leq k} \subseteq \Ai'$. Let $C_1,\dots,C_m\in\Ci$ be the left-neighbourhoods $C_i = A'_i\setminus(\cup_{j=0}^{i-1}A'_j)$. We note that $C = A\setminus(\cup_{j=0}^{m}A'_j)$ and $X_{A'_i}$ is $\Gc{C}$-measurable for every $i\in\brc{1,\dotsc,m}$. Therefore, for any $\Gamma,\Gamma_1,\dots,\Gamma_m\in\Ei$, we have
  \begin{align*}
    &\espb[_\nu]{\indi_{\Gamma_1}(X_{A'_1})\dotsb\indi_{\Gamma_m}(X_{A'_m}) \, \prc[_\nu]{X_A\in\Gamma}{\Gc{C}} } \\
    &=\prb[_\nu]{X_{A'_1}\in\Gamma_1,\dotsc,X_{A'_m}\in\Gamma_m,X_A\in\Gamma} \\
    &=\int_E \nu(\dt x_0) \brcbb{ \int_{E^{m+1}} \Pc{C_1}(x_0;\dt x_{A'_1})\dotsb \Pc{C_m}(\vxCp{C_m};\dt x_{A'_m})\,\Pc{C}(\vxCp{C};\dt x_{A}) \\
    & \phantom{=\int_E \nu(\dt x_0) \int_{E^{m+1}}\quad} \indi_{\Gamma_1}(x_{A'_1})\dotsb\indi_{\Gamma_m}(x_{A'_m}) \,\indi_{\Gamma}(x_A) }
  \end{align*}
  The integral over $x_A$ does not depend on the other terms, and is equal to $\int_E \Pc{C}(\vxCp{C};\dt x_{A})\,\indi_{\Gamma}(x_A) = \Pc{C}(\vxCp{C};\Gamma)$. Hence, 
  \begin{align*}
    &\espb[_\nu]{\indi_{\Gamma_1}(X_{A'_1})\dotsb\indi_{\Gamma_m}(X_{A'_m}) \, \prc[_\nu]{X_A\in\Gamma}{\Gc{C}} } \\
    &=\int_{E^{m+1}} \nu(\dt x_0) \Pc{C_1}(x_0;\dt x_{A'_1})\dotsb \Pc{C_m}(\vxCp{C_m};\dt x_{A'_m})\,\indi_{\Gamma_1}(x_{A'_1})\dotsb\indi_{\Gamma_m}(x_{A'_m})
\Pc{C}(\vxCp{C};\Gamma) \\
    &=\espb[_\nu]{\indi_{\Gamma_1}(X_{A'_1})\dotsb\indi_{\Gamma_m}(X_{A'_m}) \, \Pc{C}(\vXCp{C};\Gamma) }.
  \end{align*}
  A monotone class argument allows to conclude the proof of the \Ci-Markov property, i.e.
  $\prc[_\nu]{X_A\in\Gamma}{\Gc{C}} = \Pc{C}(\vXCp{C};\Gamma)\ \Pr_\nu$-almost surely for all $C\in\Ci$ and $\Gamma\in\Ei$.
\end{proof}

Common notations $\pr[_x]{\,.\,}$ and $\esp[_x]{\,.\,}$ are used later in the case of Dirac initial distributions $\nu = \delta_x$, $x\in E$. Similarly to the classic Markov theory, if $Z$ is a bounded random variable on the probability space $(\Omega,\Fi,\Pr)$, a monotone class argument shows that the map $x\mapsto\esp[_x]{Z}$ is measurable and for any probability measure $\nu$ on $(E,\Ei)$, 
\[
  \esp[_\nu]{Z} = \int_E \nu(\dt x)\, \esp[_x]{Z}.
\]

\subsection{Set-indexed Markov properties} \label{ssec:si_markov}

As outlined in the introduction, several set-indexed Markov properties have already been investigated in the literature. Hence, it seems natural to wonder if the \Ci-Markov property is related to some of them.

To begin with, let us recall the definitions of the \emph{Markov} and \emph{sharp Markov} properties presented by \citet{Ivanoff.Merzbach(2000)a}.
\begin{definition}[\textbf{Markov SI processes \cite{Ivanoff.Merzbach(2000)a}}]
  A set-indexed process $X$ is said to be \emph{Markov} if for all $B\in\Aiu$ and all $A_1,\dotsc,A_k\in\Ai$ such that $A_i\nsubseteq B$, $i\in\brc{1,\dotsc,k}$,
  \[
    \espcb{ f(X_{A_1},\dotsc,X_{A_k}) }{\Fi_B} = \espcb{ f(X_{A_1},\dotsc,X_{A_k}) }{\Fi_{\partial B} \cap \Fi_{\cup_{i=1}^k A_i} }\quad \quad\text{$\Pr$-a.s.}
  \]
  where $\Fi_{\partial B}$ denotes the $\sigma$-field $\sigma\pthb{\brc{X_A;\, A\in\Ai,A\subseteq B,A\nsubseteq B^\circ}}$ and $f$ is a measurable function $f:\R^k\rightarrow\R_+$.
\end{definition}
\begin{definition}[\textbf{sharp Markov processes \cite{Ivanoff.Merzbach(2000)a}}] \label{def:sharp_markov}
  A set-indexed process $X$ is said to be \emph{sharp Markov} if for all $B\in\Aiu$,
  \begin{equation} \label{eq:sharp_markov}
    \Fi_B \perp \Fi_{B^c} \,|\ \Fi_{\partial B},
  \end{equation}
  where $\Fi_{B^c}$ is defined as the $\sigma$-field $\sigma\pthb{\brc{ X_A;\, A\in\Ai, A\nsubseteq B}}$. We recall that for any $\sigma$-fields $\Fi_1,\Fi_2,\Fi_3$, the notation $\Fi_1 \perp \Fi_2\ |\ \Fi_3$ means that $\Fi_1$ and $\Fi_2$ are conditionally independent given $\Fi_3$. 
\end{definition}
We may also compare our definition to the \emph{set-Markov} property introduced by \citet{Balan.Ivanoff(2002)} and deeply studied in both \cite{Balan.Ivanoff(2002)} and \cite{Balan(2004)}.
\begin{definition}[\textbf{set-Markov processes \cite{Balan.Ivanoff(2002)}}] \label{def:set_markov}
  A set-indexed process $X$ is said to be \emph{set-Markov} if for all $B,B'\in\Aiu$, $B\subseteq B'$ and for any measurable function $f:\R\rightarrow\R_+$,
  \[
    \espcb{f(\Delta X_{B'})}{\Fi_B} = \espcb{f(\Delta X_{B'})}{\Delta X_B}\quad \quad\text{$\Pr$-a.s.}
  \]
\end{definition}

It has been proved in \cite{Ivanoff.Merzbach(2000)a,Balan.Ivanoff(2002)} that the aforementioned Markov properties satisfy the following implications:
\begin{center}
  \emph{set-Markov} $\Rightarrow$ \emph{sharp Markov}\ \ \ and\ \ \ \emph{Markov} $\Rightarrow$ \emph{sharp Markov}.
\end{center}
The latter is an equivalence when the filtration verifies an assumption of \emph{conditional orthogonality} (see Definition $2.3$ in \cite{Ivanoff.Merzbach(2000)a} for more details).

\begin{proposition} \label{prop:sharp_markov}
  Let $(\Omega,\Fi,(\Fi_A)_{A\in\Ai},\Pr)$ be a complete probability space and $X$ be a $\Ci$-Markov process with respect to $(\Fi_A)_{A\in\Ai}$. 
  
  Then, $X$ is also a \emph{Markov} and a \emph{Sharp Markov} process.
\end{proposition}
\begin{proof}
  Let $X$ be \emph{\Ci-Markov} on $(\Omega,\Fi,(\Fi_A)_{A\in\Ai},\Pr)$. Since the \emph{Markov} property implies \emph{Sharp Markov}, we only have to prove that $X$ is \emph{Markov}. 
  
  Let $B = \cup_{i=1}^l A^1_i\in\Aiu$ and $A_1,\dotsc,A_k\in\Ai$  such that $A_i\nsubseteq B$ for every $i\in\brc{1,\dotsc,k}$. The collection of sets $\brc{A_1,\dots,A_k,A^1_1,\dotsc,A^1_l}$ generates a semilattice $\Ai'=\brc{A'_0=\zset,A'_1,\dots,A'_m}$. Without any loss of generality, we may assume that the consistent numbering of $\Ai'$ is such that
  \[
    \exists p\leq m : \forall i\in\brc{1,\dotsc,m};\quad A'_i\nsubseteq B \Longleftrightarrow i\geq p.
  \]
  $C_1,\dotsc,C_m\in\Ci$ designate the usual left-neighbourhoods $C_i = A'_i\setminus \cup_{j=0}^{i-1}A'_j$. Note that for all $i\in\brc{p,\dotsc,m}$, $\Fi_B\subseteq G^*_{C_i}$ as $C_i\cap B=\vset$.
  Then, for any measurable function $f:\R^k\rightarrow\R_+$,
  \begin{align*}
    \espc{f(X_{A_1},\dotsc,X_{A_k})}{\Fi_B} 
    &= \espcb{ \espc{f(X_{A_1},\dotsc,X_{A_k})}{\Gc{C_m}} }{\Fi_B} \\
    &= \espcB{ \int_\R \Pc{C_m}(\vXCp{C_m};\dt x_{A'_m}) \, f(X_{A_1},\dotsc,X_{A_{k-1}},x_{A'_m}) }{\Fi_B},
  \end{align*}
  where we assume that $A'_m=A_k$. Hence, we obtain by induction
  \begin{align} \label{eq:proof_markov}
    &\espc{f(X_{A_1},\dotsc,X_{A_k})}{\Fi_B} \nonumber \\
    &= \espcB{ \int_{\R^{m-p+1}} \Pc{C_p}(\vXCp{C_p};\dt x_{A'_p}) \dotsb \Pc{C_m}(\vXCp{C_m};\dt x_{A'_m})\, f(x_{A_1},\dotsc,x_{A_k}) }{\Fi_B} \nonumber \\
    &= \int_{\R^{m-p+1}} \Pc{C_p}(\vXCp{C_p};\dt x_{A'_p}) \dotsb \Pc{C_m}(\vXCp{C_m};\dt x_{A'_m})\, f(x_{A_1},\dotsc,x_{A_k}),
  \end{align}
  since for all $i\in\brc{1,\dotsc,p-1}$, $X_{A'_i}$ is $\Fi_B$-measurable ($A'_i\subseteq B$).
  
  Consider $i\in\brc{p,\dotsc,m}$. Since $A'_i\subseteq \cup_{j=1}^k A_j$, we know that $\vXCp{C_i}$ is $\Fi_{\cup_{j=1}^k A_j}$-measurable. Furthermore, let $U\in\Cp{C_i}$ such that $U\neq A'_p,\dotsc,A'_m$. Due to the definition of $p$ and as $C_i\cap B = \vset$, $U\subseteq B$ and $U\nsubseteq B^\circ$. Hence, $X_U$ is $\Fi_{\partial B}$-measurable. 
  
  As we integrate over the variables $x_{A'_p},\dotsc,x_{A'_m}$ in Equation \eqref{eq:proof_markov}, the only random variables $X_U$ left are such that $U\in\Cp{C_i}$ and $U\neq A'_p,\dotsc,A'_m$. Therefore, the integral \eqref{eq:proof_markov} is $\Fi_{\cup_{j=1}^k A_j}$- and $\Fi_{\partial B}$-measurable, proving that
  \begin{align*}
    \espc{f(X_{A_1},\dotsc,X_{A_k})}{\Fi_B} = \espcb{ f(X_{A_1},\dotsc,X_{A_k}) }{\Fi_{\partial B} \cap \Fi_{\cup_{i=1}^k A_i} }.
  \end{align*}
\end{proof}

An interesting discussion lies in the comparison between the \emph{set-Markov} and \emph{\Ci-Markov} properties. We first observe that they both satisfy a few basic features which may be expected from any set-indexed Markov definition: they imply the \emph{sharp-Markov} property and processes with independent increments (see Example \ref{ex:indep_incr}) are both \emph{\Ci-Markov} and \emph{set-Markov}. Nevertheless, one might quickly notice that they are not equivalent. Indeed, if we simply consider the empirical process $X_A = \sum_{i=1}^n \indi_{\brc{Z_j\in A}}$, where $(Z_i)_{i\leq n}$ are i.i.d. random variables, we know from \cite{Balan.Ivanoff(2002)} that $X$ is \emph{set-Markov} whereas a simple calculus shows that it is not \emph{\Ci-Markov}. Conversely, the \emph{set-indexed \SaS Ornstein--Uhlenbeck process} presented later is a \emph{\Ci-Markov} process which is not \emph{set-Markov} (see Section \ref{ssec:si_examples}).

The \emph{\Ci-Markov} and \emph{set-Markov} properties can be seen as two consistent ways to present a set-indexed Markov property leading to the definition of a transition system, \Pii and \Qi respectively, which characterizes entirely the finite-dimensional distributions. If both \Pii and \Qi satisfy a Chapman--Kolmogorov like equation,
\begin{equation} \label{eq:chapman_kolmogorov_all}
  \Pc{C} f = \Pc{C'} \Pc{C''} f\quad\text{and}\quad Q_{BB''} f = Q_{B B'}\, Q_{B'B''}f,
\end{equation}
 an important difference lies in the indexing family used. On the one hand, $\Pii=\brc{\Pc{C}(\vxCp{C};\dt x_A); C\in\Ci}$ considers variables indexed by \Ai, whereas on the other hand, $\Qi = \brc{Q_{BB'}(x_B;\dt x_{B'});\,B,B'\in\Aiu,B\subseteq B'}$ used \Aiu as indexing collection. 

In fact, as stated in \cite{Balan.Ivanoff(2002)}, the \Qi-Markov property is defined on the global extension $\Delta X$ on the collection \Aiu. This strong feature explains why a supplementary assumption is needed on the transition system (Theorem $1$ in \cite{Balan.Ivanoff(2002)}), to ensure the existence of $\Delta X$. More precisely, the construction of \Qi-Markov process requires that the integral
\begin{equation} \label{eq:set_markov_assumption}
  \int_{\R^{m+1}} \mu(\dt x_0) \,\indi_{\Gamma_0}(x_0) \prod_{i=1}^m \indi_{\Gamma_i}(x_i - x_{i-1}) \,Q_{\cup_{j=0}^{i-1} A_j,\cup_{j=0}^{i} A_j}(x_{i-1};\dt x_i),
\end{equation}
is independent of the numbering of $\Ai'$, for any semilattice $\Ai'=\brc{A_0,\dotsc,A_m}$. Note that this last assumption is not necessary when $\Qi$ is homogeneous in a certain sense (see the work of \citet{Herbin.Merzbach(2013)}).

On the other hand, the \Ci-Markov property adopts a different approach. The indexing collection is assumed to satisfy the \emph{Shape} assumption so that any set-indexed process is well-defined on \Aiu and the definition of a \Ci-Markov process makes sense (in particular the random vector $\vXCp{C}$, $C\in\Ci$, exists). This hypothesis is not necessary if one only considers processes with independent increments, such as the set-indexed Lévy processes, since these are independently random scattered measures (see \cite{Herbin.Merzbach(2013)} for more details on IRSM). Nevertheless, one can easily check that this Shape assumption is required to define some other classes of \Ci-Markov processes, such as the \emph{set-indexed Ornstein--Uhlenbeck processes} presented in Section \ref{ssec:si_examples}. 

One might also wonder why, on the contrary to the \emph{set-Markov} property and some other set-indexed references, we assume that $\emptyset\notin\Ai$. The explanation lies in the characterization and construction Theorems \ref{th:c_markov_charac} and \ref{th:c_markov_construct}. If we assume that $\vset\in\Ai$, we observe that $\Ai\subset\Ci$, and in particular $\zset\in\Ci$. Thereby, the transition system \Pii incorporates the transition probability $\Pii_{\zset}(x_\vset,\dt x_\zset)$. But since $X_\vset=0$ is always assumed (simply for consistency), we observe that the measure $\Pii_{\zset}(x_\vset,\dt x_\zset)$ in fact corresponds to the law of $X_\zset$. This statement clearly contradicts our will to separate in Theorems \ref{th:c_markov_charac} and \ref{th:c_markov_construct} the initial distribution $\nu$ of a process from its transition system $\Pii$.\vsp

To end this section, let us finally note that when branches of a tree form the indexing collection (as presented in the introduction), \Ci-Markov processes are in fact Markov chains indexed by trees, as defined in the seminal article of \citet{Benjamini.Peres(1994)} (see also the recent work of \citet{Durand(2009)a}).

\subsection{\texorpdfstring{Features of \Ci-Markov processes}{Features of C-Markov processes}} \label{ssec:si_properties}

Following the definition of a set-indexed Markov property, several simple and natural questions arise. Among them, we study in this section the projections of \Ci-Markov processes on flows, conditional independence of filtrations and homogeneity of transition probabilities. 

To begin with, let us recall that an \emph{elementary flow} is a continuous increasing function $f:\ivff{a,b}\subset\R_+\rightarrow\Ai$, i.e.
\begin{enumerate}[\it (i)]
  \item Increasing: $\forall s,t\in\ivff{a,b};\quad s<t \Rightarrow f(s)\subseteq f(t)$;\vsp
  \item Outer-continuous: $\forall s\in\ivfo{a,b};\quad f(s) = \bigcap_{v>s}f(v)$;\vsp
  \item Inner-continuous: $\forall s\in\ivoo{a,b};\quad f(s) = \overline{\bigcup_{u<s} f(u)}$.
\end{enumerate}
As we might expect, the elementary projections of \Ci-Markov processes happen to be one-parameter Markov processes.
\begin{proposition} \label{prop:c_markov_eflow}
  Let $(\Omega,\Fi,(\Fi_A)_{A\in\Ai},\Pr)$ be a complete probability space and $X$ be a $\Ci$-Markov process w.r.t. $(\Fi_A)_{A\in\Ai}$. 
  
  Then, for any elementary flow $f:\ivff{a,b}\rightarrow\Ai$, the projection $X^f$ of $X$ along $f$ is a Markov process w.r.t. the filtration $\Fi^f = \pthb{ \Fi_{f(s)} }_{s\in\ivff{a,b}}$. Furthermore, $X^f$ has the following transition probabilities,
  \[
    \forall s\leq t\in\ivff{a,b};\quad P^f_{s,t}(x_s;\dt x_t) = \Pc{f(t)\setminus f(s)}\pthb{ x_{f(s)};\dt x_{f(t)} },
  \]
  where $\Pii$ is the \Ci-transition system of $X$.
\end{proposition}
\begin{proof}
  We simply check that $X^f$ is a Markov process with the expected transition probabilities. Let $f$ be an elementary flow and $s\leq t$ be in $\ivff{a,b}$. Note that $\Fi_{f(s)}\subseteq \Gi^*_{f(t)\setminus f(s)}$ and $\Cp{f(t)\setminus f(s)} = \brc{f(s)}$. Then, for any $\Gamma\in\Ei$,
  \begin{align*}
    \prcb{X^f_{t}\in\Gamma}{\Fi^f_s}
    &= \prcb{ \prc{X_{f(t)}\in\Gamma}{\Gi^*_{f(t)\setminus f(s)}} }{\Fi_{f(s)}} \\
    &= \espc{ \Pc{f(t)\setminus f(s)}(X_{f(s)};\Gamma) }{\Fi_{f(s)}} = P^f_{s,t}(X^f_s;\Gamma).
  \end{align*}
\end{proof}
Note that the converse result is not true: it is not sufficient to have Markov elementary projections to obtain a \Ci-Markov process. Furthermore, Proposition \ref{prop:c_markov_eflow} can not be extended to \emph{simple flows}, i.e. continuous increasing function $f:\ivff{a,b}\rightarrow\Aiu$, as this property constitutes a characterization of \emph{set-Markov} processes (Proposition 2 in \cite{Balan.Ivanoff(2002)}). \vsp

The \emph{conditional independence} of filtrations is an important property in the theory of multiparameter processes. Also named \emph{commuting property} or (F4), it has first been introduced by \citet{Cairoli.Walsh(1975)} for two-parameter processes and recently extended to the set-indexed formalism by \citet{Ivanoff.Merzbach(2000)}.
\begin{theorem}[\textbf{Conditional independence}] \label{th:c_markov_CI}
  Let $(\Omega,\Fi,(\Fi_A)_{A\in\Ai},\Pr)$ be a complete probability space and $X$ be a $\Ci$-Markov w.r.t. $(\Fi_A)_{A\in\Ai}$. Suppose that $(\Fi_A)_{A\in\Ai}$ is included in the $\sigma$-field $\Fi^0_\infty \eqdef \sigma(\brc{X_A; A\in\Ai})$.
  
  Then, the filtration $(\Fi_A)_{A\in\Ai}$ satisfies the \emph{conditional independence} property (CI), i.e. for all $U,V\in\Ai$ and any bounded random variable $Y$ on the space $(\Omega,\Fi^0_\infty,\Pr)$,
  \[
    \espcb{ \espc{Y}{\Fi_U} }{\Fi_V} = \espcb{ \espc{Y}{\Fi_V} }{\Fi_U} = \espc{Y}{\Fi_{U\cap V}}\quad\text{$\Pr$-a.s.}
  \]
\end{theorem}
\begin{proof}
  Let $U,V\in\Ai$. If $U\subseteq V$ or $V\subseteq U$, the equality is straight-forward. Hence, we suppose throughout that $U\nsubseteq V$ and $V\nsubseteq U$.
  Using a monotone class argument, we only need to prove that
  \[
    \espcb{ \espc{f_1(X_{A_1})\dotsb f_k(X_{A_k})}{\Fi_U} }{\Fi_V} = \espc{f_1(X_{A_1})\dotsb f_k(A_k)}{\Fi_{U\cap V}},
  \]
  for every $k\in\N$, any $A_1,\dotsc,A_k\in\Ai$ and any $f_1,\dotsc,f_k:E\rightarrow\R_+$ measurable functions. Without any loss of generality, we assume that $\Ai'=\brc{A_0=\zset,A_1,\dotsc,A_k}$ is a semilattice with a consistent numbering and which contains the sets $U$, $V$ and $U\cap V$. We denote as usually $(C_i)_{i\leq k}$ the left-neighbourhoods in $\Ai'$, i.e. $C_i=A_i\setminus(\cup_{j=0}^{i-1} A_j)$. Finally, up to a re-ordering of $\Ai'$, we can suppose there exist $k_{U\cap V}, k_V\in\N$ such that $k_{U\cap V} < k_V$,
  \[
    \forall i\in\brc{1,\dotsc,k_{U\cap V}}; \quad A_i\subseteq A_{k_{U\cap V}}\eqdef U\cap V \quad\text{and}\quad \forall i\in\brc{1,\dotsc,k_{V}}; \quad A_i\subseteq A_{k_{V}}\eqdef V.
  \]
  Then, as $\Fi_V\subseteq\Gc{C_k}$, the \Ci-Markov property induces
  \begin{align*}
    &\espcb{ \espc{f_1(X_{A_1})\dotsb f_k(X_{A_k})}{\Fi_V} }{\Fi_U} \\
    &= \espcb{ \espc{ \espc{ f_1(X_{A_1})\dotsb f_k(X_{A_k}) }{\Gc{C_k}} }{\Fi_V} }{\Fi_U} \\
    &= \espcbb{ \espcB{ \int_E \Pc{C_k}(\vXCp{C_k};\dt x_{A_k}) f_1(X_{A_1})\dotsb f_k(x_{A_k}) }{\Fi_V} }{\Fi_U},
  \end{align*}
  Therefore, by induction,
  \begin{align*}
    & \espcb{ \espc{f_1(X_{A_1})\dotsb f_k(X_{A_k})}{\Fi_V} }{\Fi_U} \\
    &= \espcbb{ \espcB{ \int_{E^{k-k_V}} \Pc{C_{k_V+1}}(\vXCp{C_{k_V+1}};\dt x_{A_{k_V+1}}) \dotsb \Pc{C_k}(\vXCp{C_k};\dt x_{A_k})  \\
    &\phantom{ =\int_{E^{k-k_V}} \qquad q} f_1(X_{A_1})\dotsb f_{k_V}(X_{A_{k_V}})\,f_{k_V+1}(x_{A_{k_V+1}}) \dotsb f_k(x_{A_k}) }{\Fi_V} }{\Fi_U}.
  \end{align*}
  Every random variable $X_{A_i}$ left in the previous integral is such that $A_i\subseteq V$, and thus is $\Fi_V$-measurable. Furthermore, for every $i\in\brc{k_{U\cap V}+1,\dotsc,k_V}$,
  \[
    U\cap C_i \subseteq U\cap (A_i\setminus U\cap V) = (U\cap A_i)\setminus (U\cap V) = \vset\quad \text{since } A_i\subseteq V.
  \]
  Therefore, $\Fi_U\subseteq\Gi^*_{C_i}$ and we obtain
  \begin{align*}
    & \espcb{ \espc{f_1(X_{A_1})\dotsb f_k(X_{A_k})}{\Fi_V} }{\Fi_U} \\
    &= \espcB{ \int_{E^{k-k_V}} \Pc{C_{k_V+1}}(\vXCp{C_{k_V+1}};\dt x_{A_{k_V+1}}) \dotsb \Pc{C_k}(\vXCp{C_k};\dt x_{A_k}) \\
    &\phantom{ =\int_{E^{k-k_V}} \qquad} f_1(X_{A_1})\dotsb f_{k_V+1}(x_{A_{k_V+1}}) \dotsb f_k(x_{A_k}) }{\Fi_U} \\
    &= \espcB{ \int_{E^{k-k_{U\cap V}}} \Pc{C_{k_{U\cap V}+1}}(\vXCp{C_{k_{U\cap V}+1}};\dt x_{A_{k_{U\cap V}+1}}) \dotsb \Pc{C_k}(\vXCp{C_k};\dt x_{A_k}) \\
    &\phantom{ =\int_{E^{k-k_{U\cap V}}} \qquad} f_1(X_{A_1})\dotsb f_{k_{U\cap V}}(X_{A_{k_{U\cap V}}}) f_{k_{U\cap V}+1}(x_{A_{k_{U\cap V}+1}}) \dotsb f_k(x_{A_k}) }{\Fi_U} \\
  \end{align*}
  Finally, every random variable $X_{A_i}$ left in the integral is $\Fi_{U\cap V}$-measurable (as $A_i\subseteq U\cap V$ for every $i\in\brc{1,\dotsc,k_{U\cap V}}$), proving that
  \[
    \espcb{ \espc{f_1(X_{A_1})\dotsb f_k(X_{A_k})}{\Fi_V} }{\Fi_U}
    = \espc{f_1(X_{A_1})\dotsb f_k(A_k)}{\Fi_{U\cap V}}.
  \]
\end{proof}

In the next theorem, we investigate the extension of the \emph{simple Markov property} to the \Ci-Markov formalism.
In order to introduce the concept of \emph{homogeneous \Ci-transition system}, we suppose that there exists a collection of shift operators $(\theta_U)_{U\in\Ai}$ on \Ai which satisfies:
\begin{enumerate}[\it (i)]
  \item for any $U\in\Ai$ such that $U^\circ\neq\vset$, we have $A\subset(\theta_U(A))^\circ$ for all $A\in\Ai$;
  \item for any $A\in\Ai$, $U\mapsto\theta_U(A)$ is an increasing monotone outer-continuous function with $\theta_\zset(A)=A$;
  \item for any $U\in\Ai$, $\theta_U$ is stable under intersections and $\theta_U(\zset)=U$.
\end{enumerate}
For all $U\in\Ai$, the operator $\theta_U$ is extended on \Aiu and \Ci in the following way:
\[
  \theta_U(B) = \cup_{i=1}^k \theta_U(A_i) \quad\text{and}\quad \theta_U(C) = \theta_U(A) \setminus \theta_U(B),
\]
where $B=\cup_{i=1}^k A_i\in\Aiu$ and $C=A\setminus B\in\Ci$.

To illustrate this notion, let simply consider the multiparameter setting $\Ai=\brc{{\ivff{0,t}} : t\in\R_+^N}$. In this framework, the natural family $(\theta_u)_{u\in\Ti}$ of shift operators corresponds to the usual translation on $\R^N_+$, i.e.
\[
  \forall u,t\in\R^N_+;\quad \theta_u(t) = \theta_{\ivff{0,u}}\pthb{ \ivff{0,t} } \eqdef \ivff{0,t+u}.
\]
One can easily verify that $(\theta_u)_{u\in\R^N_+}$ satisfy the previous conditions.

\begin{definition}[\textbf{Homogeneous \Ci-transition system}]
  A \Ci-transition system \Pii is said to be \emph{homogeneous} with respect to $(\theta_U)_{U\in\Ai}$ if it satisfies
  \[
    \forall U\in\Ai,\,\forall C=A\setminus B\in\Ci;\quad \Pc{C}(\vxCp{C};\dt x_A) = \Pc{\theta_U(C)}(\vxCp{ \theta_U(C) };\dt x_{\theta_U(A)} ).
  \]
  A \Ci-Markov process is said to be \emph{homogeneous} w.r.t $(\theta_U)_{U\in\Ai}$ if its \Ci-transition system is itself homogeneous.
\end{definition}

We can now establish a set-indexed \emph{simple \Ci-Markov property}. In the following theorem, $\Omega$ and $X$ refer to the canonical space $E^\Ai$ and the canonical process $X_A(\omega)=\omega(A)$, respectively.
\begin{theorem}[\textbf{Simple \Ci-Markov property}] \label{th:c_markov_simple}
  Let \Pii be an homogeneous \Ci-transition system w.r.t $(\theta_U)_{U\in\Ai}$. Then, the canonical \Ci-Markov process $X$ satisfies
  \[
    \espcb[_\nu]{f(X\circ\theta_U)}{\Fi_U} = \espb[_{X_U}]{f(X)}\quad \Pr_\nu\text{-a.s.},
  \]
  for any $U\in\Ai$, any measurable function $f:E^\Ai\rightarrow\R_+$ and any initial measure $\nu$.
\end{theorem}
\begin{proof}
  Let $A_1,\dotsc,A_k$ be in \Ai. Without any loss of generality, we can suppose that $\Ai'=\brc{A_0=\zset,A_1,\dotsc,A_k}$ is a semilattice with consistent numbering. As usually, $C_1,\dotsc,C_k\in\Ci$ denote the left-neighbourhoods $C_i=A_i\setminus(\cup_{j=0}^{i-1}A_j)$. Then, for any measurable function $h:E^k\rightarrow\R_+$ and any $U\in\Ai$,
  \begin{align*}
    &\espcb[_\nu]{h(X_{\theta_U(A_1)},\dotsc,X_{\theta_U(A_k)})}{\Fi_U} \\
    &= \espcb[_\nu]{ \espc[_\nu]{ h(X_{\theta_U(A_1)},\dotsc,X_{\theta_U(A_k)}) }{\Gc{\theta_U(C_k)}} }{\Fi_U} \\
    &= \espcB[_\nu]{\int_\R \Pc{\theta_U(C_k)}(\vXCp{\theta_U(C_k)};\dt x_{\theta_U(A_k)})\,h(X_{\theta_U(A_1)},\dotsc,X_{\theta_U(A_{k-1})},x_{\theta_U(A_k)})}{\Fi_U},
  \end{align*}
  since for every $i\in\brc{1,\dotsc,k}$, $\theta_U(C_i) \cap U = \vset$, and thus $\Fi_U\subseteq \Gc{\theta_U(C_k)}$. By induction, 
  \begin{align*}
    &\espcb[_\nu]{h(X_{\theta_U(A_1)},\dotsc,X_{\theta_U(A_k)})}{\Fi_U} \\
    &= \int_{\R^k} \Pc{\theta_U(C_1)}(X_U;\dt x_{\theta_U(A_1)})\dotsc \Pc{\theta_U(C_k)}(\vxCp{\theta_U(C_k)};\dt x_{\theta_U(A_k)})\, h(x_{\theta_U(A_1)},\dotsc,x_{\theta_U(A_k)}) \\
    &= \int_{\R^k} \Pc{C_1}(X_U;\dt x_{A_1})\dotsc \Pc{C_k}(\vxCp{C_k};\dt x_{A_k})\, h(x_{A_1},\dotsc,x_{A_k}) \\
    &= \esp[_{X_U}]{h(X_{A_1},\dotsc,X_{A_k})},
  \end{align*}
  as \Pii is homogeneous and $\theta_U(C_1) = \theta_U(A_1)\setminus\theta_U(\zset) = \theta_U(A_1)\setminus U$.
  A monotone class argument allows to conclude the proof.
\end{proof}

\subsection{\texorpdfstring{\Ci-Feller processes}{C-Feller processes}} \label{ssec:si_feller}

The Feller property plays a major role in the classic theory of Markov processes, motivating its introduction into the \Ci-Markov formalism. For this purpose, we need to reinforce the assumptions on the indexing collection \Ai, so that a pseudo-metric can be defined on the latter and on the class of increments \Ci.
\begin{definition}[\textbf{Separability from above, see \cite{Ivanoff.Merzbach(2000)}}]
  There exist an increasing sequence of finite subclasses $\Ai_n=\brc{A_1^n,\dotsc,A_{k_n}^n}$ of \Ai closed under intersections and such that $B_n\in\Ai_n$; a sequence $(\eps_n)_{n\in\N}$ decreasing to $0$ and a sequence of functions $g_n:\Ai\rightarrow \Ai_n\cup\brc{\Ti}$ such that
  
  \begin{enumerate}[\it (i)]
    \item for any $A\in\Ai,\ A = \cap_n g_n(A)$;
    \item $g_n$ preserves arbitrary intersections and finite unions;
    \item for any $A\in\Ai,\ A\subseteq(g_n(A))^\circ$;
    \item $g_n(A)\subseteq g_m(A)$ if $n \geq m$;
    \item for any $A\in\Ai$, $d_H(A,g_n(A)) \leq \eps_n$ for all $n\in\N$ such that $g_n(A)\subseteq B_n$,
  \end{enumerate}
  where $d_H$ denotes the Hausdorff distance between two sets.
\end{definition}
The family $(g_n)_{n\in\N}$ plays a role similar to the dyadic rational numbers in $\R^N$.
Note that the previous definition is slightly stronger than the usual one presented by \citet{Ivanoff.Merzbach(2000)}, since every $g_n$ is supposed to be $\Ai_n$-valued, and not $\Ai_n(u)$.
For any filtration $\pth{\Fi_A}_{A\in\Ai}$, the collection $(g_n)_{n\in\N}$ induces the definition of the augmented filtration $\pth{\widetilde\Fi_A}_{A\in\Ai}$:
\[
  \forall A\in\Ai; \quad \widetilde\Fi_A \eqdef \bigcap_{n\in\N} \Fi_{g_n(A)}.
\]
The filtrations $(\widetilde\Fi_B)_{B\in\Aiu}$ and $(\GcA{C})_{C\in\Ci}$, given by Equation \eqref{eq:def_filtrations}, denote the extensions of $\widetilde\Fi$ on \Aiu and \Ci, respectively. As proved by \citet{Ivanoff.Merzbach(2000)}, the augmented filtration is monotone outer-continuous, i.e. $\widetilde\Fi_B = \cap_{n\in\N} \widetilde\Fi_{B_n}$ for any decreasing sequence $(B_n)_{n\in\N}$ in $\Aiu$ such that $B=\cap_{n\in\N} B_n$.  

In the classic theory of Markov processes, a Feller semigroup $P$ must satisfy the condition $P_t f \rightarrow_{t\rightarrow 0} f$. Therefore, to adapt such a property to the \Ci-Markov formalism, one must define a pseudo-norm $\norm{\,\cdot\,}_{\Ci}$ on the class of increments $\Ci$.
For this purpose, let us introduce a few more notations related to \Ci. For any $C = A\setminus B\in\Ci$, we assume that the numbering $B = \cup_{i=1}^k A_i$ is such that $d_H(A,A_1) \leq \dotsc \leq d_H(A,A_k)$. Then, we define a map $\psi_C:\brc{\Cpi{C}{0},\Cpi{C}{1},\dotsc,\Cpi{C}{p}}\rightarrow\brc{\Cpi{C}{1},\dotsc,\Cpi{C}{p}}$ constructed as follows (we recall that the definition of $(\Cpi{C}{i})_{i\leq p}$ is given in the introduction):
\begin{enumerate}[ \it 1.]
  \item $\psi_C(\Cpi{C}{0}) = A_1$, where by convention $\Cpi{C}{0} = A$;
  \item $\psi_C(\Cpi{C}{1}) = A_1$, where it is assumed that $\Cpi{C}{1} = A_1$;
  \item let $i\in\brc{2,\dotsc,p}$ and $V = A_1\cap\Cpi{C}{i}$.
  Owing to Equation \eqref{eq:inc_exc_formula}, we observe that $X_{A_1} = \sum_{i=1}^{p} (-1)^{\eCpi{i}} X_{A_1\cap\Cpi{C}{i}}$, which implies the existence of $j > i$ such that $V = A_1\cap\Cpi{C}{j}$. Consequently, we set 
  \begin{equation} \label{eq:psi_C}
    \psi_C(\Cpi{C}{i}) = \Cpi{C}{j}\quad\text{and}\quad \psi_C(\Cpi{C}{j}) = \Cpi{C}{j}.
  \end{equation}
  By induction, we obtain a complete definition of $\psi_C$.
\end{enumerate}
We note that the construction of $\psi_C$ implies that $\Delta X$ satisfies
\begin{equation} \label{eq:psi_C_delta}
  \Delta X_C = \pthb{ X_A - X_{\psi_C(A)} } - \bktbb{ \sum_{i=1}^{p} (-1)^{\eCpi{i}} \pthb{ X_{\Cpi{C}{i}} - X_{\psi_C(\Cpi{C}{i})} } }.
\end{equation}
The map $\psi_C$ leads to the definition of the pseudo-norms $\norm{C}_\Ci$ and $\norm{\vxCp{C}}_\Ci$:
\begin{equation} \label{eq:norm_C}
  \norm{C}_\Ci \eqdef \max_{0\leq i \leq p} d_H\pthb{ \Cpi{C}{i}, \psi_C(\Cpi{C}{i}) } \quad\text{and}\quad \norm{\vxCp{C}}_\Ci \eqdef \max_{0\leq i \leq p} d_E\pthb{ x_{\Cpi{C}{i}} , x_{\psi_C(\Cpi{C}{i})}  }.
\end{equation}
Finally, for any $m\in\N$, let $\Ci^{m}$ denotes the sub-class $\brc{C\in\Ci : \abs{\Cp{C}} \leq m \text{ and } C\subseteq B_m }$.\vspar

\begin{definition}[\textbf{\Ci-Feller transition system}] \label{def:c_markov_feller}
  A \Ci-transition system \Pii is said to be \emph{Feller} if it satisfies the following conditions:
  \begin{enumerate}[ \it 1.]
    \item for any $f\in C_0(E)$ and for all $C\in\Ci$, 
    \[
      \Pc{C}f\in C_0\pthb{ E^{\abs{\Cp{C}}} },
    \]
    where $C_0(E^k)$, $k\in\N$, denotes the usual space of continuous functions that vanish at infinity;
    \item for any $f\in C_0(E)$ and for every $m\in\N$,
    \begin{equation} \label{eq:feller_property}
      \lim_{\rho\rightarrow 0} \ \sup_{\stackrel{C = C'\cup C''\in\Ci^{m}}{\norm{C'}_\Ci \leq \rho}} \ \sup_{\stackrel{ \vxCp{C''},\vxCp{C'} }{ \norm{ \vxCp{C'} }_\Ci \leq\rho } } \absb{  \Pc{C} f(\vxCp{C}) - P_{C''}f(\vxCp{C''}) } = 0.
    \end{equation}
    where, similarly to Definition \ref{def:c_markov_tr}, the sets $C$, $C'$ and $C''$ are such that $C'=C\cap A'$ and $C''=C\setminus A'$, with $A'\in\Ai$.
  \end{enumerate}
  A \Ci-Markov process $X$ is said to be \emph{\Ci-Feller} if its \Ci-transition system \Pii is Feller.
\end{definition}
The definition of a \Ci-Feller transition system, and particularly Equation \eqref{eq:feller_property}, may seem a bit cumbersome, but in the one-dimensional case and if \Pii is homogeneous, it is equivalent to the usual Feller property $\lim_{t\rightarrow 0} \, \sup_{s\geq 0} \, \norm{ P_{s+t}f - P_{s}f }_\infty = 0$.
  
Moreover, in the case of simple increments $C' = U\setminus V$ and $C''=\vset$, it corresponds to a more common formula:
\begin{equation} \label{eq:feller_property_weak}
  \lim_{\rho\rightarrow 0} \ \sup_{\overset{U,V\in\Ai,\, U\subseteq V\subset B_m}{d_H(U,U\cap V)\leq\rho} } \normb{\Pc{U\setminus V}f - f}_\infty = 0,
\end{equation}
since $\norm{U\setminus V}_\Ci = d_H(U,U\cap V)$ and $\norm{\vxCp{U\setminus V}}_\Ci = d_E(x_U,x_{U\cap V})$.

The existence of interesting \Ci-Feller processes which are not homogeneous (e.g. the well-known multiparameter Brownian sheet) motivates the separation between the homogeneous and Feller hypotheses, at the cost of a more complex definition of the latter (Equation \eqref{eq:feller_property}).

Finally, in this section are considered \Ci-Markov processes $X$ which have outer-continuous sample paths, meaning that:
\begin{equation} \label{eq:si_cad}
  \forall A\in\Ai,\ \forall \eps>0,\ \exists \alpha>0;\quad  A\subseteq U\in\Ai\ \text{ and }\ d_H(A,U)\leq\alpha\ \Longrightarrow\ d_E\pth{X_A,X_U} \leq \eps.
\end{equation}

We can now extend a classic result of the theory of Markov processes.
\begin{theorem}[\textbf{\Ci-Markov w.r.t the augmented filtration}] \label{th:c_markov_aug}
  Let $(\Omega,\Fi,(\Fi_A)_{A\in\Ai},\Pr)$ be a complete probability space and $X$ be a \Ci-Feller process w.r.t $(\Fi_A)_{A\in\Ai}$. In addition, suppose $X$ has outer-continuous sample paths. 
  
  Then, $X$ is a \Ci-Markov process with respect to the augmented filtration $(\widetilde\Fi_A)_{A\in\Ai}$, i.e.
  \[
    \espc{f(X_A)}{\GcA{C}} = \espc{f(X_A)}{\vXCp{C}} \quad\text{\Pr-a.s.}
  \]
  for all $C=A\setminus B\in\Ci$ and any measurable function $f:E\mapsto\R_+$.
\end{theorem}
\begin{proof}
  Let $C=A\setminus B$ be in \Ci, with $B=\cup_{j=1}^k A_j$. We assume that the numbering of $\Cp{C} = \brc{\Cpi{C}{1},\dotsc,\Cpi{C}{p}}$ is consistent. Let $\rho_0 = \max_{1\leq i<j\leq p} d_H(\Cpi{C}{i},\Cpi{C}{j}) > 0$. Then, for all $n\in\N$ and for every $i\in\brc{1,\dotsc,p+1}$, let $C^n_i$ and $D^n_i$ denote the following increments,
  \[
    C^n_i = A\setminus \pthbb{ B\cup\bigcup_{j=1}^{i-1} g_n(\Cpi{C}{j}) } \quad\text{and}\quad D^n_i = g_n(\Cpi{C}{i}) \setminus \pthbb{ B\cup\bigcup_{j=1}^{i-1} g_n(\Cpi{C}{j}) } ,
  \]
  where by convention $g_n(\Cpi{C}{p+1}) = A$. For every $i\in\brc{1,\dotsc,p}$, let $D^n_i = g_n(\Cpi{C}{i}) \setminus\pthb{ \cup_{j=1}^{k_i}  A^i_j }$ be the extremal representation of $D^n_i$. As previously said, it is assumed that the numbering is such that $d_H(g_n(\Cpi{C}{i}),A^i_1)\leq\dotsb\leq d_H(g_n(\Cpi{C}{i}),A^i_{k_i})$. For every $j$, due to the definition of $D^n_i$, there exist $l_j$ such that 
  \[
    A^i_j = g_n(\Cpi{C}{i}) \cap \Cpi{C}{l_j} \quad\text{or}\quad A^i_j = g_n(\Cpi{C}{i}) \cap g_n(\Cpi{C}{l_j}) = g_n(\Cpi{C}{i} \cap \Cpi{C}{l_j}).
  \]
  Since $\Cpi{C}{i}\subseteq B$, there exists $j\leq k_i$ such that $\Cpi{C}{i}\subseteq A^i_j \subset g_n(\Cpi{C}{i})$. Therefore, as the family $(g_n)_n$ satisfies assumption \textit{(v)}, we have $d_H(g_n(\Cpi{C}{i}),A^i_j) \rightarrow 0$ and $A^i_j$ converges to $\Cpi{C}{i}$.
  
  If we consider the particular case of $A^i_1$, we obtain that $d_H(g_n(\Cpi{C}{i}),A^i_1) \rightarrow 0$ and $A^i_1$ converges to an element $V \eqdef \Cpi{C}{i} \cap \Cpi{C}{l_1}\in\Cp{C}$. Let suppose $V\neq\Cpi{C}{i}$, implying that $d_H(V,\Cpi{C}{i}) \geq \rho_0$. We know that $d_H(A^i_1,V) \rightarrow 0$, $d_H(g_n(\Cpi{C}{i}),A^i_1) \rightarrow 0$ and $d_H(g_n(\Cpi{C}{i}),\Cpi{C}{i}) \rightarrow 0$, which is in contradiction with the previous assumption. Therefore, $\Cpi{C}{i} \cap \Cpi{C}{l_1} = \Cpi{C}{i}$.
  
  Then, let $V_n\in\Cp{D^n_i}$. According to the definition of $D^n_i$, $V_n$ must have the following general form
  \[
    V_n = \Cpi{C}{j_V}\cap g_n(\Cpi{C}{l_V}) \cap g_n(\Cpi{C}{i})\quad\text{where }j_V,l_V\in\brc{1,\dotsc,p}.
  \]
  Therefore, $(V_n)_{n\in\N}$ is a decreasing sequence which converges to $\Cpi{C}{j_V}\cap \Cpi{C}{l_V} \cap \Cpi{C}{i} \eqdef \Cpi{C}{k_V}$. Since $A^i_1$ converges to $\Cpi{C}{i}$, $V_n\cap A^i_1$ is a decreasing sequence whose limit is also $\Cpi{C}{k_V}$. Hence, according to definition \eqref{eq:psi_C} of $\psi_{D^n_i}$, we know that both $V_n$ and $\psi_{D^n_i}(V_n)$ converge to $\Cpi{C}{k_V}$, and as the family $(g_n)_n$ satisfies the assumption \textit{(v)}, for any $\rho>0$, there exists $N_\rho$ such that
  \[
    \forall n\geq N_\rho;\quad d_H\pthB{ V_n, \psi_{D^n_i}(V_n) } \leq \rho.
  \]
  Furthermore, since the sample paths of $X$ are outer-continuous, $N_\rho$ can be suppose large enough to satisfy
  \[
    \forall n\geq N_\rho;\quad d_E\pthB{ X_{V_n},X_{\Cpi{C}{k_V}} } \leq\rho \quad\text{and}\quad d_E\pthB{ X_{\psi_{ D^n_i}(V_n)} , X_{\Cpi{C}{k_v}} } \leq\rho 
  \]
  Finally, as the previous two properties are satisfied uniformly for any $V_n\in\Cp{D^n_i}$ and $i\in\brc{1,\dotsc,p}$, 
  \begin{equation} \label{eq:proof_feller}
    \forall n\geq N_\rho;\quad \normb{ D^n_i }_\Ci \leq \rho \quad\text{and}\quad \normb{ \vXCp{D^n_i} }_\Ci \leq \rho.
  \end{equation}
  Let $\eps>0$ and $f\in C_0(E)$. According to Definition \ref{def:c_markov_feller}, there exists $\rho>0$ such that
  \[
    \sup_{\stackrel{C = C'\cup C''\in\Ci^{m}}{\norm{C'}_\Ci \leq \rho}} \ \sup_{\stackrel{ \vxCp{C''},\vxCp{C'} }{ \norm{ \vxCp{C'} }_\Ci \leq\rho } } \absb{  \Pc{C} f(\vxCp{C}) - P_{C''}f(\vxCp{C''}) } \leq \eps.
  \]
  Note that for every $i\in\brc{1,\dotsc,p}$, $C^n_{i+1} = C^n_i \setminus g_n(\Cpi{C}{i})$ and $D^n_i = C^n_i \cap g_n(\Cpi{C}{i})$, and there exists $m\in\N$ such that $C^n_i$ and $D^n_i$ belong to $\Ci^m$. 
  
  Therefore, owing to the Feller property and Inequalities \eqref{eq:proof_feller},
  \[
    \forall i\in\brc{1,\dotsc,p};\quad \absb{  \Pc{C^n_i} f(\vXCp{C^n_i}) - P_{C^n_{i+1}}f(\vXCp{C^n_{i+1}}) } \leq \eps.
  \]
  Thereby, by induction and since $C^n_1 = C$, we have $\absb{  \Pc{C} f(\vXCp{C}) - P_{C^n_{p+1}}f(\vXCp{C^n_{p+1}}) } \leq p\eps$, for all $n\geq N_\rho$, which proves that $P_{C^n_{p+1}}f(\vXCp{C^n_{p+1}}) \xrightarrow[n\rightarrow\infty]{} \Pc{C} f(\vXCp{C})$ almost surely.
  
  As we observe that for all $n\in\N$, $\GcA{C}\subset\Gc{C^n_{p+1}}$, we can eventually prove the \Ci-Markov property with respect to the augmented filtration:
  \begin{align*}
    \espc{f(X_A)}{\GcA{C}}
    = \espcb{ \espc{f(X_A)}{\Gc{C^n_{k+1}}} }{\GcA{C}}
    = \espcb{ \Pc{C^n_{p+1}} f(\vXCp{C^n_{p+1}}) }{\GcA{C}}.
  \end{align*}
  By the dominated convergence theorem,
  \begin{align*}
    \espc{f(X_A)}{\widetilde\Gi^*_C}
    = \lim_{n\rightarrow\infty} \espcb{ \Pc{C^n_{p+1}} f(\vXCp{C^n_{p+1}}) }{\GcA{C}} 
    = \espc{ \Pc{C} f(\vXCp{C})}{\GcA{C}} 
    = \Pc{C} f(\vXCp{C}).
  \end{align*}
  since $\vXCp{C}$ is $\GcA{C}$-measurable.
\end{proof}

Before establishing a \emph{strong \Ci-Markov property}, we recall the definition of a \emph{simple stopping set}, as presented by \citet{Ivanoff.Merzbach(2000)}.
A \emph{simple stopping set} is an random variable $\xi:\Omega\rightarrow\Ai$ such that for all $A\in\Ai$, $\brc{\omega : A\subseteq\xi(\omega)}\in\Fi_A$. $\xi$ is said to be bounded if there exists $V\in\Ai$ such that $\xi\subseteq V$ almost surely. As the assumption $\emph{Shape}$ holds on the collection \Ai, we can define the $\sigma$-algebra of the events prior to $\xi$:
\[
  \widetilde\Fi_\xi = \brcb{ F\in\Fi :F\cap\brc{\xi\subseteq B}\in\widetilde\Fi_B, \ \forall B\in\Aiu }.
\]
According to \cite{Ivanoff.Merzbach(2000)}, if $X$ has outer-continuous sample paths, then $X_\xi$ is $\widetilde\Fi_\xi$-measurable.

\begin{theorem}[\textbf{Strong \Ci-Markov property}] \label{th:c_markov_strong}
  Let $(\Omega,\Fi,(\Fi_A)_{A\in\Ai},\Pr)$ be a complete probability space, $X$ be a $\Ci$-Feller process w.r.t. $(\Fi_A)_{A\in\Ai}$ and $\xi$ be a bounded simple stopping set. Moreover, suppose $X$ has outer-continuous sample paths and is homogeneous w.r.t. $(\theta_U)_{U\in\Ai}$.
  
  Then, $X$ satisfies a \emph{strong \Ci-Markov property}:
  \[
    \espcb[_\nu]{f\pth{X\circ\theta_\xi}}{\widetilde\Fi_\xi} = \espb[_{X_\xi}]{f(X)}\quad \Pr_\nu\text{-a.s.}
  \]
  for any initial measure $\nu$ and any measurable function $f:E^\Ai\rightarrow\R_+$.
\end{theorem}
\begin{proof}
  Let $F$ be in $\widetilde\Fi_\xi$, $\Ai' = \brc{A_0=\zset,A_1,\dotsc,A_k}$ be a finite semilattice, $C_1,\dotsc,C_k\in\Ci$ be the left-neighbourhoods $C_i=A_i\setminus (\cup_{j=0}^{i-1} A_j)$, $i\in\brc{1,\dotsc,k}$.
  
  According to Lemma $1.5.4$ in \cite{Ivanoff.Merzbach(2000)}, for every $n\in\N$, $g_n(\xi)$ is a discrete simple stopping set such that $\xi\subset g_n(\xi)$, $\xi = \cap_n g_n(\xi)$ and $\widetilde\Fi_\xi\subseteq\widetilde\Fi_{g_n(\xi)}$. Hence, $F\in\widetilde\Fi_{g_n(\xi)}$ and for any $h\in C_0(E)$,
  \[
    \esp[_\nu]{ \indi_F \, h(X_{\theta_\xi(A_k)})} = \lim_{n\rightarrow\infty}  \esp[_\nu]{ \indi_F \, h(X_{\theta_{g_n(\xi)}(A_k)})},
  \]
  since $X$ and $U\mapsto\theta_{U}$ are outer-continuous. Furthermore, $g_n(\xi)$ is a discrete stopping set and $F\cap\brc{g_n(\xi)=U}\in\widetilde\Fi_U\subset\GcA{\theta_U(C_k)}$. Hence, 
  \begin{align*}
    \esp[_\nu]{ \indi_F \, h(X_{\theta_\xi(A_k)})} 
    &= \lim_{n\rightarrow\infty} \sum_{U\in\Ai_n} \espb[_\nu]{ \indi_{F\cap\brc{g_n(\xi)=U}} \, h(X_{\theta_{U}(A_k)})} \\
    &= \lim_{n\rightarrow\infty} \sum_{U\in\Ai_n} \espb[_\nu]{ \indi_{F\cap\brc{g_n(\xi)=U}} \,\espc[_\nu]{h(X_{\theta_{U}(A_k)})}{\GcA{\theta_U(C_k)}} }.
  \end{align*}
  Owing to the \Ci-Markov property and the homogeneity of \Pii,
  \begin{align*}
    \esp[_\nu]{ \indi_F \, h(X_{\theta_\xi(A_k)})} 
    &= \lim_{n\rightarrow\infty} \sum_{U\in\Ai_n} \espb[_\nu]{ \indi_{F\cap\brc{g_n(\xi)=U}} \,  \Pc{\theta_U(C_k)}h(\vXCp{\theta_U(C_k)}) } \\
    &= \lim_{n\rightarrow\infty} \sum_{U\in\Ai_n} \espb[_\nu]{ \indi_{F\cap\brc{g_n(\xi)=U}} \,  \Pc{C_k}h(\vXCp{\theta_U(C_k)}) } \\
    &= \lim_{n\rightarrow\infty} \esp[_\nu]{ \indi_{F} \,  \Pc{C_k}h(\vXCp{\theta_{g_n(\xi)}(C_k)}) }.
  \end{align*}
  Finally, since $\Pc{C_k}h\in C_0(E^{\abs{\Cp{C_k}}})$ and $X$ is outer-continuous, the dominated convergence theorem leads to
  \[
    \esp[_\nu]{ \indi_F \, h(X_{\theta_\xi(A_k)})} = \esp[_\nu]{ \indi_{F} \, \Pc{C_k}h(\vXCp{\theta_{\xi}(C_k)}) }.
  \]
  A simple induction argument extends the equality to any collection $h_1,\dotsc,h_k$ in $C_0(E)$,
  \begin{align*}
    &\esp[_\nu]{ \indi_F \, h_1(X_{\theta_\xi(A_1)})\dotsb h_k(X_{\theta_\xi(A_k)}) } \\
    &= \esp[_\nu]{ \indi_F \, h_1(X_{\theta_\xi(A_1)})\dotsb h_{k-1}(X_{\theta_\xi(A_{k-1})})\, \Pc{C_k}h_k(\vXCp{\theta_{\xi}(C_k)}) } \\
    &= \espbb[_\nu]{\indi_F \int_{\R^k} \Pc{C_1}(X_\xi;\dt x_{A_1})\dotsc \Pc{C_k}(\vxCp{C_k};\dt x_{A_k})\, h_1(x_{A_1})\dotsc h_k(x_{A_k})},
  \end{align*}
  using $\theta_\xi(\zset) = \xi$. This last formula proves that
  \[
    \espb[_\nu]{\indi_F\, h_1(X_{\theta_\xi(A_1)}) \dotsb h_k(X_{\theta_\xi(A_k)})} = \espb[_\nu]{\indi_F\, \esp[_{X_\xi}]{h_1(X_{A_1})\dotsb h_k(X_{A_k})} },
  \]
  which leads to the expected result, by a monotone class argument.
\end{proof}

The outer-continuity of sample paths is an important assumption in Theorems \ref{th:c_markov_aug} and \ref{th:c_markov_strong}. In the general set-indexed framework, the question of the regularity of processes is known to be non-trivial, and far more complex than in the multiparameter setting. For instance, \citet{Adler(1981)} has presented indexing collections on which the set-indexed Brownian motion is not continuous (and even not bounded). There exist several approach in the literature to study this question, including the metric entropy theory \cite{Adler(1981),Adler.Feigin(1984)} or more recently, the set-indexed formalism \cite{Herbin.Richard(2012)}. 

The c\`adl\`aguity of \Ci-Markov processes is studied in Section \ref{sec:mpCMarkov} within the multiparameter setting.

\subsection{Applications and examples} \label{ssec:si_examples}

To conclude this general study of the \Ci-Markov property, we present several examples of \Ci-Markov processes which are natural extensions of classic one- and multi-parameter processes.

\begin{example}[\textbf{Processes with independent increments}] \label{ex:indep_incr}
  A set-indexed process $X$ has independent increments if for any disjoints sets $C_1,\dotsc,C_k\in\Ci$, $\Delta X_{C_1},\dotsc,\Delta X_{C_k}$ are independent random variables. Or equivalently, if for any $C\in\Ci$, the increment $\Delta X_C$ is independent of the $\sigma$-field $\Gc{C}$. 
  
  Let $P_{X_C}$ denote the law of the increment $\Delta X_C$. Then, for all $C=A\setminus B\in\Ci$ and any $\Gamma\in\Ei$,
  \begin{align*}
    \prc{X_A\in\Gamma}{\Gc{C}}
    &= \prc{\Delta X_C\in\Gamma - \Delta X_B}{\Gc{C}} \\
    &= P_{X_C}(\Gamma - \Delta X_B) \eqdef \Pc{C}(\vXCp{C};\Gamma) \quad \text{ owing to Equation \eqref{eq:inc_exc_formula}.}
  \end{align*}
  Therefore, as we might have expected, set-indexed processes with independent increments satisfy the \Ci-Markov property. As proved by \citet{Balan.Ivanoff(2002)}, they are also \emph{set-Markov}.
\end{example}

\begin{example}[\textbf{Set-indexed L\'evy processes}] \label{ex:si_levy}
  Set-indexed L\'evy processes have been defined and studied by \citet{Adler.Feigin(1984),Bass.Pyke(1984)} in the particular case of subsets of $\R^N$, and by \citet{Herbin.Merzbach(2013)} in the general set-indexed formalism. In the latter, a set-indexed L\'evy process $X$ on a measure space $(\Ti,m)$ is characterized as follows:
  \begin{enumerate}[ \it 1)]
    \item $X_{\zset} = 0$ almost surely;
    \item $X$ has independent increments;
    \item $X$ has m-stationary \Cio-increments, i.e. for all $V\in\Ai$ and for all increasing sequences $(U_i)_{i\leq n}$ and $(A_i)_{i\leq n}$ in \Ai such that $m(U_i\setminus V)=m(A_i)$,
    \[
      \pthb{ \Delta X_{U_1\setminus V},\dotsc,\Delta X_{U_n\setminus V} } \eqd \pthb{ \Delta X_{A_1},\dotsc,\Delta X_{A_n} };
    \]
    \item $X$ is continuous in probability.
  \end{enumerate}
  According to \cite{Herbin.Merzbach(2013)}, for any set-indexed L\'evy process $X$, there exists an infinitely divisible measure $\mu$ on $(E,\Ei)$ such that
  for all $C\in\Ci$, $P_{X_C} = \mu^{m(C)}$.
  Since $X$ has independent increments, it is a \Ci-Markov process with the following \Ci-transition system \Pii, 
  \begin{equation} \label{eq:ex_levy_kernel}
    \forall C=A\setminus B\in\Ci,\quad\forall\Gamma\in\Ei;\quad \Pc{C}(\vxCp{C};\Gamma) = \mu^{m(C)}(\Gamma - \Delta x_B).
  \end{equation}
  Conversely, one can easily prove, using the \Ci-Markov property, that a L\'evy process can be constructed from any infinitely divisible measure $\mu$. For this purpose, one simply needs to show that \Pii satisfies the three conditions of Definition \ref{def:c_markov_tr}.
  \begin{enumerate}[ \it 1)]
    \item For any $C\in\Ci$, $\Pc{C}(\vxCp{C};\dt x_A)$ is clearly a transition probability;
    \item As $m(\vset)=0$, we have $\Pc{\vset}(x;\dt y) = \delta_x(\dt y)$;
    \item For any $C = A\setminus B\in\Ci$ and $A'\in\Ai$, let $C' = C\cap A'$ and $C'' = C\setminus A'$ in \Ci. Then, we observe that 
    \[
      \mu^{m(C)} = \mu^{m(C')} \ast \mu^{m(C'')} \quad\text{ and }\quad x_{A'} + \Delta x_{B} = \Delta x_{A'\cup B} + \Delta x_{A'\cap B},
    \]
    using the equality $m(C)=m(C')+m(C'')$ and the inclusion-exclusion principle. These two equations induce the Chapman--Kolmogorov formula \eqref{eq:chapman_kolmogorov}.
  \end{enumerate}
  Note that the construction procedure presented by \citet{Herbin.Merzbach(2013)} is completely equivalent to the proof of Theorem \ref{th:c_markov_construct} (in fact more general in the case of L\'evy processes since the \emph{Shape} assumption is not required on the collection \Ai).
  
  One can also check that the \Ci-transition system is homogeneous w.r.t $(\theta_U)_{U\in\Ai}$ if and only if the measure $m$ is compatible with the family operators, i.e. $\forall U\in\Ai,\ \forall C\in\Ci;\ m(C) = m(\theta_U(C))$.
  
  Furthermore, if the measure $m$ and the Hausdorff metric $d_H$ satisfy the following mild condition,
  \begin{equation}  \label{eq:feller_assumption}
    \forall m\in\N\quad \exists K_m > 0\quad\forall C\in\Ci^m;\quad m(C) \leq K_m \norm{C}_\Ci,
  \end{equation}
  then the \Ci-transition system \Pii of a L\'evy process is Feller. Indeed, suppose $f\in C_0(E)$ and $C = C'\cup C'' \in\Ci$. Then,
  \begin{equation}  \label{eq:feller_proof1}
    \Pc{C}f(\vxCp{C}) = \int_E \mu^{m(C'')}(\dt x_{A}) \brcbb{ \int_E  f(x_A + x_{A'} + \Delta x_B) \,\mu^{m(C')}(\dt x_{A'}) }.
  \end{equation}
  Due to the Feller property of one-parameter L\'evy processes,
  \[
    \lim_{\rho\rightarrow 0}\ \sup_{\overset{C'\in\Ci}{m(C')\leq\rho}} \sup_{y\in E} \ \absbb{ \int_E f(x_{A'} + y)\,\mu^{m(C')}(\dt x_{A'}) - f(y) } = 0.
  \]
  Furthermore, if $\norm{\vxCp{C'}}\leq\rho$, Equation \eqref{eq:psi_C_delta} implies that $\abs{\Delta x_{C'}} \leq m\rho$. Hence, since $\Delta x_{B''} = \Delta x_{C'} + \Delta x_{B}$ and $f\in C_0(E)$,
  \[
    \lim_{\rho\rightarrow 0}\ \sup_{\overset{C'\in\Ci}{m(C')\leq\rho}}\ \sup_{\stackrel{ \vxCp{C''},\vxCp{C'} }{ \norm{ \vxCp{C'} }_\Ci \leq\rho } }\ \absbb{ \int_E f(x_A + x_{A'} + \Delta x_B)\,\mu^{m(C')}(\dt x_{A'}) - f(x_A + \Delta x_{B''}) } = 0.
  \]
  Assumption \eqref{eq:feller_assumption}, Equation \eqref{eq:feller_proof1} and the last equality induce the Feller property:
  \[
    \lim_{\rho\rightarrow 0} \ \sup_{\stackrel{C = C'\cup C''\in\Ci^{m}}{\norm{C'}_\Ci \leq \rho}} \ \sup_{\stackrel{ \vxCp{C''},\vxCp{C'} }{ \norm{ \vxCp{C'} }_\Ci \leq\rho } } \absb{  \Pc{C} f(\vxCp{C}) - P_{C''}f(\vxCp{C''}) } = 0.  
  \]
  Note that in the particular case of the set-indexed Brownian motion, \Pii is characterized by the following transition densities
  \[
    \forall C=A\setminus B\in\Ci;\quad \pc{C}(\vxCp{C};y) = \frac{1}{\pthb{ 2\pi \,m(C) }^{d/2}} e^{-\frac{\norm{y - \Delta x_B}_2^2}{2m(C)}}.
  \]
  Finally, we observe that Theorem \ref{th:c_markov_CI} extends the Cairoli-Walsh Commutation Theorem (see e.g. in \cite{Walsh(1986)}) to the set-indexed formalism, showing that the Brownian motion history is a commuting filtration.
\end{example}

\begin{example}[\textbf{Set-indexed $\alpha$-stable Ornstein--Uhlenbeck process}] \label{ex:si_ou}
  The classic Ornstein--Uhlenbeck (OU) process is a well-known stochastic process which has the following integral representation (see e.g. \cite{Samorodnitsky.Taqqu(1994)}):
  \begin{equation} \label{eq:stable_ou}
    X_t = \int_{-\infty}^t e^{-\lambda(t-u)} M(\dt u),
  \end{equation}
  where $\lambda>0$ and $M$ is symmetric $\alpha$-stable random measure ($\alpha\in\ivof{0,2}$) with Lebesgue control measure. 
  
  A set-indexed extension of this process can not be directly deduced from Equation \eqref{eq:stable_ou}, Nevertheless, as $X$ is also a Markov process, a \Ci-transition system which generalizes the OU Markov kernel can be introduced. 
  
  More precisely, on the space $(E,\Ei)=(\R,\Bi(\R))$, a \emph{set-indexed $\alpha$-stable Ornstein--Uhlenbeck} process is defined as the \Ci-Markov process with any initial distribution $\nu$ and the \Ci-transition system \Pii characterized by 
  \begin{equation} \label{eq:siou_trans}
    \forall \vxCp{C}\in E^{\abs{\Cp{C}}}; \quad \Pc{C}f(\vxCp{C}) = \espBB{ f\pthbb{ \sigma_C X + e^{-\lambda m(A)}\bktbb{ \sum_{i=1}^{\abs{\Cp{C}}} (-1)^{\eCpi{i}}\, x_{\Cpi{C}{i}}\, e^{\lambda m(\Cpi{C}{i})} }  } }
  \end{equation}
  where $\alpha\in\ivof{0,2}$, $X$ is a symmetric $\alpha$-stable variable $S_\alpha(1,0,0)$, $f: E\rightarrow\R_+$ is a measurable function and $\sigma_C$ is defined by
  \begin{equation} \label{eq:siou_trans2}
    \sigma_C^\alpha = \frac{\sigma^\alpha}{\alpha\lambda} \pthbb{ 1 - e^{-\alpha\lambda m(A)} \bktbb{ \sum_{i=1}^{\abs{\Cp{C}}} (-1)^{\eCpi{i}} e^{\alpha\lambda m(\Cpi{C}{i})} } }.
  \end{equation}
  Let us prove that \Pii is a well-defined \Ci-transition system.
  \begin{enumerate}[ \it 1)]
    \item For any $C\in\Ci$, $\Pc{C}(\vxCp{C};\dt x_A)$ is a transition probability;\vsp
    \item $\Pc{\vset}(x;\dt y) = \delta_x(\dt y)$ since $\sigma_\vset=0$;\vsp
    \item for any $C = A\setminus B\in\Ci$ and $A'\in\Ai$, let $C' = C\cap A'$ and $C'' = C\setminus A'$ in \Ci.
    Let us show that for any measurable function $f$,
    $\Pc{C}f(\vxCp{C}) = \int_{E^2} \Pc{C'}(\vxCp{C'};d x_{A'}) \, \Pc{C''}(\vxCp{C''};\dt x_A) f(x_A)$. 
    Owing to Equation \eqref{eq:siou_trans},
    \begin{align*}
      \Pc{C''}f(\vxCp{C''}) = \espBB{ f\pthbb{ \sigma_{C''} X + e^{-\lambda m(A)}\bktbb{ \sum_{i=2}^{\abs{\Cp{C''}}} (-1)^{\eCpi{i}}\, x_{\Cpi{C''}{i}}\, e^{\lambda m(\Cpi{C''}{i})} } + x_{A'} e^{-\lambda m(A\setminus A')} } },
    \end{align*}
    where the numbering of $\Cp{C''}$ is assumed to satisfy $\Cpi{C''}{1} = A'$. Suppose $\widetilde X$ is random variable $S_\alpha(1,0,0)$ independent of $X$. Then,
    \begin{align*}
      \Pc{C'}\Pc{C''}f &= \espBB{ f\pthbb{ \sigma_{C''} X + e^{-\lambda m(A)}\bktbb{ \sum_{i=2}^{\abs{\Cp{C''}}} (-1)^{\eCpi{i}''}\, x_{\Cpi{C''}{i}}\, e^{\lambda m(\Cpi{C''}{i})} } \\
      &\qquad+ \brcbb{  \sigma_{C'} \widetilde X + e^{-\lambda m(A')}\bktbb{ \sum_{i=1}^{\abs{\Cp{C'}}} (-1)^{\eCpi{i}'}\, x_{\Cpi{C'}{i}}\, e^{\lambda m(\Cpi{C'}{i})} }  } e^{-\lambda m(A\setminus A')} } }.
    \end{align*}
    Let $h:\Ai\rightarrow \R$ be the deterministic map $h(A)=x_A e^{\lambda m(A)}$. Owing to Equation \eqref{eq:inc_exc_formula}, it has an additive extension $\Delta h$ on \Aiu which satisfies:
    \[
      \Delta h(B\cup A') = \sum_{i=1}^{\abs{\Cp{C''}}} (-1)^{\eCpi{i}''}\, x_{\Cpi{C''}{i}}\, e^{\lambda m(\Cpi{C''}{i})}
      \quad\text{and}\quad\Delta h(B\cap A') = \sum_{i=1}^{\abs{\Cp{C'}}} (-1)^{\eCpi{i}'}\, x_{\Cpi{C'}{i}}\, e^{\lambda m(\Cpi{C'}{i})}
    \]
    Furthermore, using the inclusion-exclusion principle,
    \[
      \Delta h(B\cup A') - h(A') + \Delta h(B\cap A') = \Delta h(B) = \sum_{i=1}^{\abs{\Cp{C}}} (-1)^{\eCpi{i}}\, x_{\Cpi{C}{i}}\, e^{\lambda m(\Cpi{C}{i})}.
    \]
    Therefore,
    \begin{align*}
      \Pc{C'}\Pc{C''}f(\vxCp{C}) = \espB{ f\pthB{ \sigma_{C''} X + \sigma_{C'} e^{-\lambda m(A\setminus A')} \widetilde X + e^{-\lambda m(A)}\Delta h(B) } }.
    \end{align*}
    
    Since $X$ and $\widetilde X$ are two independent \SaS variables, we know that $\sigma_{C''} X + \sigma_{C'} e^{-\lambda m(A\setminus A')} \sim S_\alpha(\sigma,0,0)$, where 
    \[
      \sigma^\alpha = \sigma_{C'}^\alpha + \sigma_{C''}^\alpha e^{-\alpha \lambda m(A\setminus A')}.
    \] 
    According to the Definition \eqref{eq:siou_trans2} of $\sigma_C$ and the inclusion-exclusion principle \eqref{eq:inc_exc_formula}, we obtain $\sigma = \sigma_C$, therefore proving the Chapman--Kolmogorov equation $\Pc{C}f = \Pc{C'}\Pc{C''}f$.
  \end{enumerate}
  Hence, we have shown that such a \Ci-Markov process exists. One can easily verify that it corresponds to the Markov kernel of the usual Ornstein--Uhlenbeck process in the case of the one-parameter indexing collection $\Ai=\brc{\ivff{0,t}; t\in\R_+}$.
  
  The Gaussian particular case $\alpha=2$ is studied in more details in \cite{Balanca.Herbin(2012)a}. Its kernel \Pii is characterized by the following transition densities
  \begin{equation*}
    \pc{C}(\vxCp{C};y) = \frac{1}{\sigma_C\sqrt{2\pi}} \exp\bktBB{ -\frac{1}{2\sigma_C^2} \pthbb{ y - e^{-\lambda m(A)}\bktbb{ \sum_{i=1}^{\abs{\Cp{C}}} (-1)^{\eCpi{i}}\, x_{\Cpi{C}{i}}\, e^{\lambda m(\Cpi{C}{i})} } }^2 },
  \end{equation*}
  If $\nu$ is the Dirac distribution $\delta_x$, $x\in\R$, then $X$ is Gaussian process such that for all $U,V\in\Ai$
  \begin{align*}
    \esp[_x]{X_U} = x e^{-\lambda m(U)} \quad\text{and}\quad \cov_x(X_U,X_V) = \frac{\sigma^2}{2\lambda}\pthb{ e^{-\lambda m(U\disy V)} - e^{-\lambda( m(U)+m(V) )} }.
  \end{align*}
  Even though the form of the transition probabilities exhibited in Equations~\eqref{eq:siou_trans} and~\eqref{eq:siou_trans2} may not seem very intuitive, \citet{Balanca.Herbin(2012)a} have adopted a more constructive presentation of the Gaussian Ornstein--Uhlenbeck process which helps to understand our present \Ci-Markov definition.
  
  Finally, similarly to the set-indexed L\'evy processes, we observe that \Pii is homogeneous when the measure $m$ is compatible with the family of operators $(\theta_U)_{U\in\Ai}$ and Feller under the mild condition \eqref{eq:feller_assumption}.
\end{example}

\begin{example}[\textbf{Set-indexed additive Lévy and product processes}] \label{ex:si_prod_add}
  From the previous examples of \Ci-Markov processes, more complex objects can be constructed. Named \emph{product} and \emph{additive Lévy} set-indexed processes, they still satisfy the \Ci-Markov property and constitute a generalization of the multiparameter examples presented by \citet{Khoshnevisan(2002)}.

  For this purpose, we first need to introduce the idea of \emph{product indexing collections}, previously mentioned in the Introduction. Given $m$ indexing collections $\Ai_1,\dotsc,\Ai_m$ on respectively $\Ti_1,\dotsc,\Ti_m$, we define the product indexing collection $\Ai$ as follows:
  \begin{equation}  \label{eq:prod_ind}
    \Ai = \brcb{A_1\times\dotsb\times A_m : A_i\in\Ai_i, \ 1\leq i\leq m},
  \end{equation}
  on the space $\Ti = \Ti_1\times\dotsb\times\Ti_m$. Conditions $(i)-(iii)$ of the definition of an indexing collection can be easily checked. 
  The \emph{Shape assumption} is slightly more difficult to verify. Up to an induction reasoning on $m$, we can assume that $m=2$. Then, let $B=\cup_{j=1}^k (A_{1,j}\times A_{2,j})$ be an element of $\Aiu$ and $A=A_1\times A_2\subseteq B$ belong to \Ai. Let also $I_1 = \brc{j : A_1 \subseteq A_{1,j}}$ and $I_2 = \brc{j : A_2 \subseteq A_{2,j}}$. We have to prove that $I_1\cap I_2\neq\vset$. Suppose on the contrary that $I_1\cap I_2=\vset$ and set 
  \[
    t_1\in A_1\setminus\pthb{\cup_{j\notin I_1} A_{1,j}}\quad\text{and}\quad t_2\in A_2\setminus\pthb{\cup_{j\notin I_2} A_{2,j}}.
  \]
  The existence of $(t_1,t_2)$ is due to the Shape assumption on $\Ai_1$ and $\Ai_2$. Then, let $j\in\brc{1,\dots,k}$. If $j\notin I_1$ and $j\notin I_2$, then $t_1\notin A_{1,j}$ and $t_2\notin A_{2,j}$. Moreover, if $j\in I_1$, then $j\notin I_2$ and thus $t_2\notin A_{2,j}$. Similarly, if $j\in I_2$, then $t_1\notin A_{1,j}$. Therefore, $(t_1,t_2)\in A_1\times A_2$ but $(t_1,t_2)\notin B$, which contradicts the hypothesis $A\subseteq B$. Hence, $I_1\cap I_2\neq \vset$ and $\Ai$ satisfies the Shape assumption.

  A class of approximating functions $(g_n)_{n\in\N}$ on \Ai can be easily deduced from existing ones on $\Ai_1,\dotsc,\Ai_m$. 
  Moreover, any family of filtrations indexed by $\Ai_1,\dotsc,\Ai_m$ leads to an \Ai-indexed extension defined as follows:
  \[
    \forall A=A_1\times\dotsb\times\Ai_m\in\Ai;\quad \Fi_A = \bigvee_{i=1}^m \Fi^i_{A_i}.
  \]
  Throughout, for any $C=A\setminus B\in\Ci$, $C_1,\dots,C_m$ will denote the following increments:
  \begin{equation}  \label{eq:incr_nota}
    C_i = A_i \setminus \pthb{\cup_{j=1}^k A_{i,j}} \in\Ci_i,\quad 1\leq i\leq m.
  \end{equation}
  Note that $C$ does not correspond to the direct product of $C_1,\dotsc,C_m$.
  
  \vsp

  Based on this construction procedure, we are able to introduce larger and richer classes of \Ci-Markov processes by combining previous examples.

  Let us first describe the family of \emph{product \Ci-Markov processes}. Given $m$ independent \Ci-Markov processes $X^1,\dotsc,X^m$, we define the product process $X$ as the direct product of the latter:
  \begin{equation}  \label{eq:def_prod_proc}
    \forall A=A_1\times\dotsb\times A_m \in\Ai;\quad X_A = X^1_{A_1}\otimes\dotsb\otimes X^m_{A_m}.
  \end{equation}
  Owing to the independence of $X^1,\dotsc,X^m$, a monotone class argument shows that $X$ is a \Ci-Markov process. Moreover, its transition probabilities are given by
  \begin{equation}  \label{eq:trans_prod}
    \forall C\in\Ci;\quad \Pc{C}(\vxCp{C};\dt x_A) = \prod_{i=1}^m \Pce{C_i}{i}(\vxCp{C_i};\dt x_{A_i}),
  \end{equation}
  where $C_i$ is the increment defined in Equation~\eqref{eq:incr_nota} and $\Pce{C_i}{i}(\vxCp{C_i};\dt x_{A_i})$ denotes the transition probabilities of $X^i$, $1\leq i \leq m$. A slightly technical calculus also shows that if $X^1,\dotsc,X^m$ are \Ci-Feller processes, then $X$ is also \Ci-Feller.

  Note that, still using the notation introduced in Equation \eqref{eq:incr_nota}, any increment of $X$ can be written as
  \[
    \forall C\in\Ci;\quad \Delta X_C = \Delta X^1_{C_1}\otimes\dotsb\otimes \Delta X^m_{C_m}.
  \]
  Particularly, if $X^1,\dotsc,X^m$ are processes with independent increments, then so does $X$.\vsp

  A second class of processes we are interested in is the family of \emph{set-indexed additive Lévy processes}. Suppose that $X^1,\dotsc,X^m$ are $\R^d$-valued set-indexed Lévy processes, characterized by $(\nu_1,m_1),\dotsc,(\nu_m,m_m)$, respectively. Then, $X$ is defined as
  \begin{equation}  \label{eq:def_add_levy}
    \forall A = A_1\times\dotsb\times A_m \in\Ai;\quad X_A = X^1_{A_1} +\dotsb+ X^m_{A_m}.
  \end{equation}
  Similarly to product processes, increments of $X$ have the following form,
  \[
    \forall C\in\Ci;\quad \Delta X_C = \Delta X^1_{C_1} +\dotsb+ \Delta X^m_{C_m},
  \]
  showing in particular that $X$ has independent increments. Its transition probabilities can be obtained as well:
  \begin{equation}  \label{eq:trans_add}
    \forall C\in\Ci;\quad \Pc{C}(\vxCp{C};\Gamma) = \pthB{\nu_1^{m_1(C_1)}\ast\dotsb\ast\nu_m^{m_m(C_m)}}(\Gamma - \Delta x_B),
  \end{equation}
  where $\Gamma\in\Bi(\R^d)$. The transition system $\Pii$ is Feller if the mild assumption \eqref{eq:feller_assumption} holds on $m_1,\dotsc,m_m$. It is homogeneous as well if $X^1,\dotsc,X^m$ are.

  Note that $X$ is not necessarily a \emph{set-indexed L\'evy process}, as defined previously in Example~\ref{ex:si_levy}. Nevertheless, a sufficient condition is the existence of a measure $\nu$ and $\alpha_1,\dotsc,\alpha_N$ in \Rpe such that $\nu = \nu_1^{\alpha_1} = \dotsb = \nu_N^{\alpha_N}$. In this case, we observe that $X$ is the Lévy process characterised by $(\nu,m_\alpha)$, where the measure $m_\alpha$ is defined as follows:
  \begin{equation}  \label{eq:add_levy_meas}
    \forall A=A_1\times\dotsb\times A_m;\quad m_\alpha(A) = \sum_{i=1}^m \alpha_i \,m_i(A_i).
  \end{equation}
  Conversely, set-indexed Lévy processes are usually not additive processes (e.g. the well-known multiparameter Brownian sheet), since in general, a measure $m$ does not have the form displayed in Equation \eqref{eq:add_levy_meas}.
\end{example}

\section{\texorpdfstring{Multiparameter \Ci-Markov processes}{Multiparameter C-Markov property}} \label{sec:mpCMarkov}

The extension of the one-parameter Markov property to the multiparameter setting has been intensively investigated and a large literature already exists on the subject. Since the set-indexed formalism covers multiparameter processes, it is therefore natural to study more precisely the \Ci-Markov property in this setting. 

Consequently, in this section, $\Ti$ and $\Ai$ designate $\R^N_+$ and the indexing collection $\brc{{\ivff{0,t}}; t\in\R_+^N}$, respectively. As previously noted, the natural translation on $\R^N_+$ leads to the definition of the shift operators $\theta_u(t) = \theta_{\ivff{0,u}}\pthb{ \ivff{0,t} } \eqdef \ivff{0,t+u}$, for all $u,t\in\R^N_+$.

\subsection{\texorpdfstring{Right-continuous modification of \Ci-Feller processes}{Right-continuous modification of C-Feller processes}}

We begin by extending a classic result of the theory of Markov processes: multiparameter \Ci-Feller processes have a right-continuous modification which is \Ci-Markov with respect to the augmented filtration. 

In the rest of the section, $E_\Delta$ denotes the usual one-point compactification of $E$, i.e. the set $E\cup\brc{\Delta}$ endowed with the following topology: $A\subset E_\Delta$ is open if either $A\subset E$ is open in the topology of $E$ or there exists a compact $K\subset E$ such that $A = E_\Delta\setminus K$.

\begin{theorem}[\textbf{C\`ad modification}] \label{th:mpcad}
  Let $\pth{X_t}_{t\in\R_+^N}$ be a multiparameter process, $\pth{\Fi_t}_{t\in\R^N_+}$ be a filtration and $(\Pr_x)_{x\in E}$ be a collection of probability measures such that for every $x\in E$, $\pr[_x]{X_0=x}=1$ and $X$ is a \Ci-Feller process on $(\Omega,\Fi,\Pr_x)$ w.r.t. $\pth{\Fi_t}_{t\in\R^N_+}$.
  
  \noindent Then, there exists an $E_\Delta$-valued process $\pth{\widetilde{X}_t}_{t\in\R^N_+}$ such that for every $x\in E$:\vspace*{-0.8em}
  \begin{enumerate}[ \it (i)]
    \item $t\mapsto \widetilde{X}_t$ is right-continuous $\Pr_x$-a.s.;
    \item for all $t\in\R_+$, $\widetilde{X}_t = X_t$ $\Pr_x$-a.s.;
    \item $\widetilde{X}$ is a \Ci-Feller process on $(\Omega,\Fi,\Pr_x)$ w.r.t. the augmented filtration $\pth{\widetilde\Fi_t}_{t\in\R^N_+}$.
  \end{enumerate}
\end{theorem}
\begin{proof}
  Let $\Hi =\brc{\varphi_n;\, n\in\N}$ be a countable collection of functions in $C_0^+(E)$ which separate points in $E_\Delta$, i.e., for any $x,y\in E_\Delta$, there exists $\varphi\in\Hi$ such that $\varphi(x)\neq \varphi(y)$.
  For all $r\in\Q^N_+$ and any $\varphi\in\Hi$, the process $M^r$ is defined as follows:
  \[
    \forall t\preceq r\in\R^N_+;\quad M^r_t = \Pc{\ivff{0,r}\setminus\ivff{0,t}} \varphi(X_t).
  \]
  For every $x\in E$, $M^r$ is multiparameter martingale with respect to the natural filtration $\pth{\Fi^0_t}_{t\in\R^N_+}$ of $X$ (i.e. $\Fi^0_t \eqdef \sigma(\brc{X_s; s \preceq t, s\in\R_+^N})$):
  \begin{align*}
    \forall t\prec t'\preceq r\in\R^N_+;\quad  
    \espc[_x]{M^r_{t'}}{\Fi^0_{t}}
    &= \espc[_x]{ \Pc{\ivff{0,r}\setminus\ivff{0,t'}} \varphi(X_{t'}) }{\Fi^0_{t}} \\
    &= \Pc{\ivff{0,t'}\setminus\ivff{0,t}}\pthb{ \Pc{\ivff{0,r}\setminus\ivff{0,t'}} \varphi }(X_t) \\
    &= \Pc{\ivff{0,r}\setminus\ivff{0,t}} \varphi(X_t) = M^r_t,
  \end{align*}
  using the Chapman--Kolmogorov Equation \eqref{eq:chapman_kolmogorov}.
  
  Owing to Theorem \ref{th:c_markov_CI}, the filtration $\pth{\Fi^0_t}_{t\in\R^N_+}$ is commuting. Then, as stated by \cite{Khoshnevisan(2002)} (and originally proved in \cite{Bakry(1979)}), since $M^r$ is a bounded multiparameter martingale with respect to a commuting filtration,  there exists an event $\Lambda_{\varphi,r}\in\Fi$ such that for every $x\in E$, $\pr[_x]{\Lambda_{\varphi,r}^c} = 0$ and for all $\omega\in\Lambda_{\varphi,r}$,
  \begin{equation} \label{eq:cad_mg}
    \lim_{s\downarrow t,\, s\in\Q_+^N} M^r_s(\omega) \text{ exists for all }t\prec r\in\R^N_+. 
  \end{equation}
  Let $\Lambda$ denote the event defined by
  \[
    \Lambda = \bigcap_{\varphi\in\Hi,r\in\Q_+^N} \Lambda_{\varphi,r}.
  \]
  Since the union is countable, $\pr[_x]{\Lambda^c}=0$ for every $x\in E$.
  
  Let us prove $\lim_{s\downarrow t,\,s\in\Q^N_+} X_s(\omega)$ exists for all $\omega\in\Lambda$. 
  We proceed by contradiction. Suppose there exist $\omega\in\Lambda$, $t\in\R_+^N$, and two decreasing sequences $(s^1_n)_{n\in\N},(s^2_n)_{n\in\N}$ in $\Q^N_+$ such that $\lim_{n\in\N} s^1_n = t$, $\lim_{n\in\N} s^2_n = t$ and
  \[
    X_t^1(\omega)\eqdef\lim_{s^1_n\downarrow t} X_{s^1_n}(\omega) \neq \lim_{s^2_n\downarrow t} X_{s^2_n}(\omega)\eqdef X_t^2(\omega),
  \]
  where the two limits stand in $E_\Delta$.
  
  Let $\varphi$ be a separating function in $\Hi$ such that $\eps = \abs{\varphi(X^1_t(\omega)) - \varphi(X^2_t(\omega))} > 0$ and $(u_n)_{n\in\N}$ be the decreasing sequence $u_n = (s^1_n\curlyvee s^2_n) + \frac{1}{n} \in\Q^N_+$.
  
  Since the transition system \Pii of $X$ is Feller, Equation \eqref{eq:feller_property_weak} implies the existence of $\alpha>0$ such that
  \[
    \forall u\preceq v\preceq 2t\in\R_+^N;\quad  \norm{u-v}_\infty\leq\alpha \quad\Longrightarrow\quad \normb{\Pc{\ivff{0,v}\setminus\ivff{0,u}}\varphi - \varphi }_\infty \leq \frac{\eps}{8}.
  \]
  Furthermore, since $u_n\rightarrow_n t$, there exists $k\in\N$ such that for all $n\geq k$, $\norm{u_n - t}\leq\alpha$, $\norm{u_k - s^1_n}\leq\alpha$ and $\norm{u_k - s^2_n}\leq\alpha$. 
  Therefore, for all $n\geq k$,
  \[
    \normb{\Pc{\ivff{0,u_k}\setminus\ivff{0,s^1_n}}\varphi - \varphi }_\infty \leq \frac{\eps}{8}\quad\text{and}\quad \normb{\Pc{\ivff{0,u_k}\setminus\ivff{0,s^2_n}}\varphi - \varphi }_\infty \leq \frac{\eps}{8}.
  \]
  Owing to Equation \eqref{eq:cad_mg} and as $\omega\in\Lambda_{\varphi,u_k}$, $k$ can be chosen large enough such that for all $n\geq k$, we have $\abs{M^{u_k}_{s_n^1}(\omega) - M^{u_k}_{s_n^2}(\omega)} \leq \tfrac{\eps}{8}$, i.e.
  \[
    \forall n\geq k; \quad \absb{ \Pc{\ivff{0,u_k}\setminus\ivff{0,s^1_n}}\varphi(X_{s^1_n}(\omega)) - \Pc{\ivff{0,u_k}\setminus\ivff{0,s^2_n}}\varphi(X_{s^2_n}(\omega)) } \leq \frac{\eps}{8}.
  \]
  Finally, as the function $\varphi$ is continuous, $k$ can be supposed to satisfy
  \[
    \forall n\geq k;\quad \absb{ \varphi(X^1_t(\omega)) - \varphi(X_{s^1_n}(\omega)) } \leq\frac{\eps}{8}\quad\text{and}\quad\absb{ \varphi(X^2_t(\omega)) - \varphi(X_{s^2_n}(\omega)) } \leq\frac{\eps}{8}.
  \]
  Due to previous inequalities, for all $n\geq k$,
  \begin{align*}
    \absb{\varphi(X^1_t(\omega)) - \varphi(X^2_t(\omega))} &\leq 
    \absb{ \varphi(X^1_t(\omega)) - \varphi(X_{s^1_n}(\omega)) } \\
    &+ \absb{ \varphi(X_{s^1_n}(\omega)) - \Pc{\ivff{0,u_k}\setminus\ivff{0,s^1_n}}\varphi(X_{s^1_n}(\omega)) } \\
    &+ \absb{ \Pc{\ivff{0,u_k}\setminus\ivff{0,s^1_n}}\varphi(X_{s^1_n}(\omega)) - \Pc{\ivff{0,u_k}\setminus\ivff{0,s^2_n}}\varphi(X_{s^2_n}(\omega)) } \\
    &+ \absb{ \varphi(X_{s^2_n}(\omega)) - \Pc{\ivff{0,u_k}\setminus\ivff{0,s^2_n}}\varphi(X_{s^2_n}(\omega)) } \\
    &+ \absb{ \varphi(X^2_t(\omega)) - \varphi(X_{s^2_n}(\omega)) } \\
    &\leq \frac{5}{8}\eps,
  \end{align*}
  which clearly contradicts the definition of $\eps$.
  
  Hence, for all $\omega\in\Lambda$ and all $t\in\R^N_+$, $\lim_{s\downarrow t,\,s\in\Q^N_+} X_s(\omega)$ exists in $E_\Delta$ and the process $\widetilde{X}$ can be defined as follows:
  \[
    \forall t\in\R^N_+;\quad \widetilde{X}_t = 
    \begin{cases}
    \lim_{s\downarrow t,\,s\in\Q^N_+} X_s(\omega) & \text{if } \omega\in\Lambda \\
    x_0  & \text{if } \omega\notin\Lambda,
    \end{cases}
  \]
  where $x_0$ is an arbitrary point in $E$. This new process $\widetilde{X}$ is clearly an $E_\Delta$-valued process with right-continuous sample paths.
  
  Let us now prove $\widetilde X$ is a modification of $X$. Let $\varphi_1$ and $\varphi_2$ be in $C_0(E)$. Then, for every $x\in E$ and all $t\in\R^N_+$, the dominated convergence theorem and the Feller property induce
  \begin{align*}
    \espb[_x]{\varphi_1(\widetilde{X}_t) \varphi_2(X_t)}
    &= \lim_{s\downarrow t,\, s\in\Q^N_+} \espb[_x]{\varphi_1(X_s) \varphi_2(X_t)} \\
    &= \lim_{s\downarrow t,\, s\in\Q^N_+} \espb[_x]{ \espc[_x]{ \varphi_1(X_s) }{\Fi_t} \,\varphi_2(X_t)} \\
    &= \lim_{s\downarrow t,\, s\in\Q^N_+} \espb[_x]{ \Pc{\ivff{0,s}\setminus\ivff{0,t}}\varphi_1(X_t) \,\varphi_2(X_t)}
    = \espb[_x]{\varphi_1(X_t) \varphi_2(X_t)}.
  \end{align*}
  A classic monotone class argument extends this equality to any measurable function $f:E^2\rightarrow\R_+$, and in particular shows that
  \[
    \forall x\in E,\ \forall t\in\R^N_+;\quad \pr[_x]{ X_t = \widetilde{X}_t } = 1.
  \]
  The process $\widetilde X$ is clearly adapted to the filtration $(\widetilde\Fi_t)_{t\in\R^N_+}$ and using Theorem \ref{th:c_markov_aug}, we obtain the \Ci-Markov property with respect to the augmented filtration.
\end{proof}
Note that the martingale argument used in the previous proof can not be directly transpose to the set-indexed formalism since there does not exist any result on set-indexed martingales stating the existence of a right-continuous modification (see~\cite{Ivanoff.Merzbach(2000)} for more details on the theory of set-indexed martingales).

\subsection{\texorpdfstring{\Ci-Markov and multiparameter Markov properties}{C-Markov and multiparameter Markov properties}}

A large literature exists on multiparameter Markov properties, especially in the case of two-parameter processes. It includes the works \cite{Levy(1948),McKean(1963),Merzbach.Nualart(1990),Dalang.Walsh(1992),Dalang.Walsh(1992)a} on the previously mentioned \emph{sharp-Markov} and \emph{germ-Markov} properties, but also the literature on Gaussian Markov random fields \cite{Pitt(1971),Kunsch(1979),Carnal.Walsh(1991)} and strong Markov properties \cite{Evstigneev(1982),Evstigneev(1988),Merzbach.Nualart(1990),Kinateder(2000)}. Different surveys and books \cite{Rozanov(1982),Dozzi(1989),Dozzi(1991),Imkeller(1988)} have been written on these distinct topics. In the context of this article, we focus on two existing multiparameter properties which can be directly linked up to our work. \vsp

To begin with, let us mentioned the two-parameter \emph{$\ast$-Markov} property introduced by \citet{Cairoli(1971)}. A process $(X_{s,t})_{s,t}$ is said to be \emph{$\ast$-Markov} if for all $s,t\in\R_+$, $h,k>0$ and any $\Gamma\in\Ei$,
\begin{equation} \label{eq:star_markov}
  \prc{ X_{s+h,t+k} \in\Gamma }{\Gi^*_{s,t}} = \prc{ X_{s+h,t+k} \in\Gamma }{ X_{s,t},X_{s+h,t},X_{s,t+k} },
\end{equation}
where $\Gi^*_{s,t}$ correspond to the strong history previously introduced (Equation \eqref{eq:def_filtrations}). The $\ast$-Markov property has been widely studied and considered in the literature, leading to several interesting results, including a martingale characterization in the Gaussian case \cite{Nualart.Sanz(1979)}, the equivalence of several different definitions \cite{Korezlioglu.Lefort.ea(1981)}, some properties on transition probabilities \cite{Luo(1988)} and the c\`adl\`aguity of sample paths \cite{Zhou.Zhou(1993)}. We note that the $\ast$-Markov property also appears in the study of a two-parameter Ornstein--Uhlenbeck process \cite{Wang(1988),Zhang(1985)}.

As stated in the following proposition, it turns out that \Ci-Markov and $\ast$-Markov properties are equivalent in the two-parameter formalism.
\begin{proposition} \label{prop:c_markov_star}
  Let $\Ti=\R^2_+$, \Ai be the indexing collection $\brc{\ivff{0,t} ; t\in\R_+^2}$ and $(X_{s,t})_{s,t}$ be a two-parameter process.

  Then, $X$ is a \emph{$\ast$-Markov} process if and only if it satisfies the \emph{\Ci-Markov} property.
\end{proposition}
\begin{proof}
  Let $X$ be a two-parameter \Ci-Markov process, $s,t\in\R_+$, $h,k>0$ and $\Gamma\in\Ei$. Define the increment $C =A\setminus B$ as $\ivff{0,(s+h,t+k)}\setminus\pthb{ \ivff{0,(s,t)}\cup\ivff{0,(s+h,t)}\cup\ivff{0,(s,t+k)} }$. We observe that $\vXCp{C} = (X_{s,t},X_{s+h,t},X_{s,t+k})$ and $\Gi^*_{s,t} = \Gc{C}$. Hence, the \Ci-Markov property $\prc{X_A\in\Gamma}{\Gc{C}} = \prc{X_A\in\Gamma}{\vXCp{C}}$ corresponds to $\ast$-Markov's definition (Equation \eqref{eq:star_markov}).
  
  Conversely, let $X$ be a $\ast$-Markov process and $C = A\setminus B$ be an increment where $A=\ivff{0,(s,t)}$ and $B = \cup_{i=1}^n \ivff{0,(s_i,t_i)}$. Without any loss of generality, we can assume that $s_1 < s_2 < \dotsc < s_n$. Since we consider an extremal representation of $B$, it must also satisfy $t_1 > t_2>\dotsc>t_n$. Therefore, $\Cp{C}$ has the following form
  \[
    \Cp{C} = \brcb{ \ivff{0,(s_1,t_1)},\ivff{0,(s_1,t_2)},\ivff{0,(s_2,t_2)},\dotsc,\ivff{0,(s_{n-1},t_n)},\ivff{0,(s_n,t_n)} }.
  \]
  The conditional expectation w.r.t. to such a collection of random variables $\vXCp{C}$ has already been considered in the work of \citet{Korezlioglu.Lefort.ea(1981)}, and the \Ci-Markov property is therefore a direct consequence of Theorem $3.7$ obtained in the latter.
\end{proof}
Proposition \ref{prop:c_markov_star} shows that the \Ci-Markov property offers an elegant way to extend the two-parameter $\ast$-Markov property to multiparameter processes, and more largely, to set-indexed processes. 

Moreover, the set-indexed formalism allows to simplify notations and concepts compared to the $\ast$-Markov framework. Indeed, as stated in \cite{Luo(1988),Zhou.Zhou(1993)}, the concept of $\ast$-transition function $(P^1,P^2,P)$ has been introduced to characterize the law of $\ast$-Markov process. The corresponding notion of \Ci-transition system \Pii replaces the triplet $(P^1,P^2,P)$, and reduces the different consistency hypotheses on $(P^1,P^2,P)$ to the single Chapman--Kolmogorov Equation \eqref{eq:chapman_kolmogorov}.
We also note that Theorem \ref{th:mpcad} extends the regularity result obtained by \citet{Zhou.Zhou(1993)} on two-parameter $\ast$-Markov processes.\vsp

A second interesting Markov property is the \emph{multiparameter Markov} property presented by \citet{Khoshnevisan(2002)} and widely studied in the literature (see e.g. \cite{Hirsch.Song(1995)a,Mazziotto(1988),Wong(1989),Khoshnevisan(2002)}).

An $E_\Delta$ valued process $X=\brc{X_t;\, t\in\R^N_+}$ is said to be a \emph{multiparameter Markov} process if there exists a filtration $(\Fi_t)_{t\in\R^N_+}$ and a family of operators $\EuScript{T}=\brc{\EuScript{T}_t;\,t\in\R^N_+}$ such that for every $x\in E$, there exists a probability measure $\Pr_x$ which satisfies
\begin{enumerate}[ \it (i)]
  \item $X$ is adapted to $(\Fi_t)_{t\in\R^N_+}$;
  \item $t\mapsto X_t$ is right-continuous $\Pr_x$-a.s.;
  \item for all $t\in\R^N_+$, $\Fi_t$ is $\Pr_x$-complete; moreover, $(\Fi_t)_{t\in\R^N_+}$ is a commuting filtration, i.e. for all $s,t\in\R^N_+$ and for any bounded $\Fi_t$-measurable random variable $Y$,
  \[
    \espc[_x]{Y}{\Fi_s} = \espc[_x]{Y}{\Fi_{s\curlywedge t}} \quad\Pr_x\text{-a.s.}
  \]
  \item for all $s,t\in\R^N_+$ and for any $f\in C_0(E)$, 
  \[
    \espc[_x]{f(X_{t+s})}{\Fi_s} = \EuScript{T}_t f(X_s) \quad\text{$\Pr_x$-a.s.}
  \]
  \item $\Pr_x(X_0=x)=1$.
\end{enumerate}
Furthermore, $X$ is said to be a \emph{Feller} process if
\begin{enumerate}[ \it (i)]
  \item for all $t\in\R^N_+$, $\EuScript{T}_t:C_0(E)\rightarrow C_0(E)$;
  \item for any $f\in C_0(E)$, $\lim_{t\rightarrow 0} \norm{\EuScript{T}_t f - f}_\infty = 0$. 
\end{enumerate}

\begin{proposition} \label{prop:c_markov_multi}
  Let $(\Pr_x)_{x\in E}$ be a family of probabilities and $X$ be a multiparameter homogeneous \Ci-Feller process w.r.t $(\Fi_t)_{t\in\R^N_+}$. Suppose $X$ has right-continuous sample paths $\Pr_x$-a.s., $\Pr_x(X_0=x)=1$  for every $x\in E$ and $\Fi_t\subseteq\sigma(\brc{X_t; t\in\R^N_+})$ for all $t\in\R_+^N$.

  Then, $X$ is a \emph{multiparameter Markov and Feller process} w.r.t. the filtration $(\Fi_t)_{t\in\R^N_+}$ and with the following transition operators $\EuScript{T}=\{\EuScript{T}_t; t\in\R^N_+\}$:
  \[
    \forall t\in\R^N_+,\ \forall f:\R\rightarrow\R_+\text{ measurable};\quad \EuScript{T}_t f = \Pc{\ivff{0,t}\setminus\{0\}} f,
  \]
  where $\Pii$ denotes the Markov kernel of $X$.
\end{proposition}
\begin{proof}
  We have to verify the different points of the previous definition of a \emph{multiparameter Markov process}.
  \begin{enumerate}[ \it (i)]
    \item $X$ is clearly adapted to $(\Fi_t)_{t\in\R^N_+}$;
    \item $X$ is right-continuous $\Pr_x$-a.s.;
    \item $(\Fi_t)_{t\in\R^n_+}$ is a commuting filtration according to Theorem \ref{th:c_markov_CI};
    \item for all $s,t\in\R^N_+$ and for any $f\in C_0(E)$,
    \begin{align*}
      \espc[_x]{f( X_{t+s})}{\Fi_s}
      &= \Pc{\ivff{0,t+s}\setminus\ivff{0,s}}f( X_{s}) && \text{since $X$ is \Ci-Markov} \\
      &= \Pc{\ivff{0,t}\setminus\{0\}}f( X_{s}) && \text{since \Pii is homogeneous} \\
      &\eqdef \EuScript{T}_t f( X_s);
    \end{align*}
    \item for all $x\in\R$, $\pr[_x]{ X_0=x}=1$.
  \end{enumerate}
  Finally, the Feller conditions on $\EuScript{T}$ are easily verified using Equation \eqref{eq:feller_property_weak} and Definition \ref{def:c_markov_feller} of a Feller \Ci-transition system.
\end{proof}
As stated by \citet{Khoshnevisan(2002)}, the family of operators $\EuScript{T}$ characterizes the two-dimensional marginals of $X$, but it is still unknown whether it determines entirely the law of $X$. Proposition \ref{prop:c_markov_multi} partially answers this question since when $X$ is also \Ci-Markov, we know that the larger family~\Pii of transition probabilities completely characterizes the finite-dimensional distributions of $X$.

Since it is not known whether the multiparameter Markov property gives a complete picture of the law of a process $X$, it is not possible to obtain a general result which would state that multiparameter Markov processes are also \Ci-Markov. Nevertheless, we show in the next examples that usual multiparameter Markov processes investigated in the literature are also \Ci-Markov.

\begin{example}[\textbf{Additive L\'evy and product processes}]
  As observed by \citet{Khoshnevisan(2002)}, \emph{multiparameter additive L\'evy} and \emph{product processes} constitute interesting and rich collections of multiparameter Markov processes. Their generic set-indexed form has already been described in Example~\ref{ex:si_prod_add}, Section~\ref{ssec:si_examples}. Hence, we only make a few remarks specific to the multiparameter setting. 

  The general \Ci-Markov characterization of product processes (Equation~\eqref{eq:trans_prod}) stands as well for \emph{multiparameter product processes}. Hence, let us simply illustrate this class by a two-parameter process named the \emph{bi-Brownian motion} and defined as
  \[
    \forall t\in\R^2_+;\quad Y_t = B^1_{t_1} \otimes B^2_{t_2}
  \]
  where $B^1$ and $B^2$ are two independent $\R^d$-valued Brownian motion. Owing to Equation~\eqref{eq:trans_prod}, the Markov kernel of $Y$ is characterized by the following transition densities, for all $C\in\Ci$,
  \[
    p_C(\vxCp{C};(u\otimes v)) = \frac{1}{(2\pi)^d (\lambda_1(C_1) \lambda_1(C_2))^{d/2}} \exp\pthbb{-\frac{\norm{u-\Delta x^1_{B_1}}^2}{2 \lambda_1(C_1)} -\frac{\norm{v-\Delta x^2_{B_2}}^2}{2 \lambda_1(C_2)}},
  \]
  where $\lambda_1$ denotes the Lebesgue measure on $\R$ and $C_1, C_2$ are the increments defined in Equation~\eqref{eq:incr_nota}. The family of transition operators $\EuScript{T}$ can be easily retrieve from this last equality.\vsp

  The common definition of an \emph{additive L\'evy process} is usually slightly more restrictive than Equation~\eqref{eq:def_add_levy}, as it usually refers to a process $X = \brc{X_t;\,t\in\R^N_+}$ defined as follows:
  \[
    \forall t\in\R_+^N;\quad X_t = X^1_{t_1} + \dotsb + X^N_{t_N}.
  \]
  where $t=(t_1,\dotsc,t_N)$ and $X^1,\dotsc,X^N$ are $N$ independent one-parameter L\'evy processes. The form of its transition probabilities still corresponds to Equation~\eqref{eq:trans_add}, where every measure $m_i$, $1\leq i\leq N$, is nevertheless replaced by the Lebesgue measure $\lambda_1$ on $\R$.
  As previously noted, this class is far from containing all multiparameter L\'evy processes, since for instance, the Brownian sheet is not an additive process.
\end{example}

\begin{example}[\textbf{Multiparameter Ornstein--Uhlenbeck processes}]
  The \emph{set-indexed Gaussian Ornstein--Uhlenbeck} process presented in \cite{Balanca.Herbin(2012)a} has an integral representation in the multiparameter setting. More precisely,
  \[
    \forall t\in\R_+^N;\quad X_t = e^{-\scpr{\alpha}{t}}\bktbb{X_0 + \sigma\int_{\ivof{-\infty,t}\setminus\ivof{-\infty,0}} e^{\scpr{\alpha}{u}}\dt W_u },
  \]
  where $\sigma > 0$, $\alpha=(\alpha_1,\dots,\alpha_N)\in\R^N$ with $\alpha_i > 0$, $W$ is the Brownian sheet and $X_0$ is a random variable independent of $W$. As previously outlined in Example \ref{ex:si_ou}, $X$ satisfies the \Ci-Markov property.

  A different, but also natural, multiparameter extension of the Ornstein--Uhlenbeck process has been suggested in the literature (e.g. see \cite{Wang(1988), Wang(1995)} and \cite{Graversen.Pedersen(2011)}). It is defined by
  \begin{equation*}
    \forall t\in\R_+^N;\quad Y_t = e^{-\scpr{\alpha}{t}}\bktbb{Y_0 + \sigma\int_{\ivff{0,t}}e^{\scpr{\alpha}{u}}\dt W_u }.
  \end{equation*}
  As proved by \citet{Wang(1995)}, the two-parameter process $Y$ is $\ast$-Markov. A calculus similar to Proposition $2.5$ from \cite{Balanca.Herbin(2012)a} shows that it is also a multiparameter \Ci-Markov process.
  Moreover, suppose $\Pii$ and $\widetilde\Pii$ denote the \Ci-transition systems of $X$ and $Y$, respectively. Then, theses two are related by the following equality:
  \[
    \widetilde P_C(\vxCp{C};\dt x_A) = 
    \begin{cases}
      \,\delta_{\Delta x_B}(\dt x_A) & \text{for all }\ C\in\Ci\ \text{ s.t. }\ C\subseteq S;\\
      \,P_C(\vxCp{C};\dt x_A) &  \text{for all }\ C\in\Ci\ \text{ s.t. }\ C\cap S=\vset,
    \end{cases}
  \]
  where $S$ denotes the set $\brc{t\in\R^N_+ : \prod_{i=1}^N t_i = 0}$.
\end{example}

\section*{Acknowledgements}
The author would like to thank his supervisor Erick Herbin and the anonymous referee whose careful proofreadings and useful comments have helped to greatly improve this paper.


\end{document}